\documentclass{amsart}

\usepackage{amssymb}
\usepackage{dsfont}

\newtheorem{theorem}{Theorem}[section]

\newtheorem{proposition}[theorem]{Proposition}
\newtheorem{corollary}[theorem]{Corollary}
\newtheorem{lemma}[theorem]{Lemma}

\theoremstyle{definition}
\newtheorem{definition}[theorem]{Definition}

\theoremstyle{remark}
\newtheorem*{remark}{Remark}
\newtheorem*{notation}{Notation}

\numberwithin{equation}{section}

\newcommand{\myskip}{\vspace{9pt}}
\newcommand{\nsubset}{\not\subset}
\newcommand{\enproof}{\hspace*{\stretch{1}}\qedsymbol}  

\newcommand{\cont}{\mathcal{I}}  
\newcommand{\norm}[1]{\|#1\|}  
\newcommand{\Lone}{\mathit{L}^1}

\newcommand{\linfty}{\ell^\infty}

\newcommand{\C}{\mathcal{C}}
\newcommand{\set}[1]{\left\{#1\right\}}  

\newcommand{\complexs}{\mathds{C}}    

\newcommand{\reals}{\mathds{R}}    
\newcommand{\naturals}{\mathds{N}}    
\newcommand{\integers}{\mathds{Z}}    

\newcommand{\unitization}[1]{#1^{\#}} 
\newcommand{\unitA}{\unitization{A}}  
\newcommand{\tuple}[1]{\boldsymbol{#1}}
\newcommand{\xiota}{x^{(P)}}

\newcommand{\setofprimes}{\mathfrak{P}}
\newcommand{\setofunions}{\mathfrak{Q}}

\newcommand{\fsetofunions}{\mathfrak{G}}
\newcommand{\quotoel}[2]{\mbox{$#1$$:$$#2$}}
\newcommand{\unit}{\mathbf{e}}

\newcommand{\F}{\mathcal{F}}

\newcommand{\G}{\mathcal{G}}

\newcommand{\U}{\mathcal{U}}

\newcommand{\abs}[1]{\left|#1\right|}
\newcommand{\infinitesimals}[1]{\left(#1\right)^\circ}
\newcommand{\ultrapower}{\complexs^{\kappa}/\U}   
\newcommand{\cardinal}[1]{\left|#1\right|}
\newcommand{\variety}{\mathcal{V}}
\newcommand{\primeradical}[1]{\sqrt{#1}}
\newcommand{\usingCH}{{\rm(CH)}\phantom{ii}}  
\newcommand{\continuum}{\mathfrak{c}}

\newcommand{\onepointcompactification}[1]{#1^\flat}
\providecommand{\rad}{\mathop{\rm rad}} 

\newcommand{\order}{\mathop{{\rm o}}}

\providecommand{\rad}{\mathop{\rm rad}} 
\newcommand{\zero}{\textbf{Z}}

\newcommand{\orderlevel}[2]{#1^{(#2)}}
\newcommand{\level}[2]{\partial^{(#2)}#1}

\begin{document}

\title
{The kernels and continuity ideals of homomorphisms from $\C_0(\Omega)$}

\dedicatory{In memoriam: Graham Robert Allan, 1937-2007}

\author{Hung Le Pham}
\address{Department of Mathematical and Statistical Sciences,
    University of Alberta, Edmonton, Alberta T6G 2G1, Canada}
\email{hlpham@math.ualberta.ca}

\thanks{This research is supported
by a Killam Postdoctoral Fellowship and a Honorary PIMS
Postdoctoral Fellowship}



\begin{abstract}
We give a description of the continuity ideals and the kernels of
homomorphisms from the algebras of continuous functions on locally
compact spaces into Banach algebras.
\end{abstract}

\maketitle

\section{Introduction}
\label{fip_introduction}

Let $\theta :A\to B$ be a homomorphism from a commutative Banach
algebra $A$ into a Banach algebra $B$. The \emph{continuity ideal}
of $\theta$ is defined to be the ideal
\[
    \cont(\theta) = \set{a\in A:\ \textrm{the map}\ b\mapsto \theta(ab),\ A\to B, \ \textrm{is
continuous}};
\]
this ideal contains every ideal $I$ in $A$ on which $\theta$ is
continuous. Moreover, in the case where $A=\C_0(\Omega)$, for a
locally compact space $\Omega$, then $\theta$ is continuous on
$\cont(\theta)$.

This is a continuation of \cite{pham2005b}. Here, we aim to
characterize the ideals which are the kernels or the continuity
ideals of homomorphisms from $\C_0(\Omega)$ into Banach algebras.
This is, in some sense, a last piece of the picture of homomorphisms
from $\C_0(\Omega)$ into Banach algebras. So what do we know about
these objects so far?

Denoted by $\abs{\,\cdot\,}_\Omega$ the uniform norm on $\Omega$.
\emph{For brevity, we shall call a homomorphism into a radical
Banach algebra a radical homomorphism}.

It is a theorem of Kaplansky \cite{kaplansky1949} that, for each
algebra norm $\norm{\cdot\,\!}$ on $\C_0(\Omega)$ and each
$f\in\C_0(\Omega)$, we have $\norm{f}\ge\abs{f}_\Omega$. This
essentially gives the description of all continuous homomorphisms
from $\C_0(\Omega)$ into Banach algebras. It was immediately asked
\cite{kaplansky1949} whether discontinuous homomorphisms from
$\C_0(\Omega)$ exist. In 1970s, this question was resolved in the
positive independently by Dales \cite{dales1979} and Esterle
\cite{esterle1978}, \cite{esterle1978b}, \cite{esterle1978c}.
Moreover, they showed that, assuming the Continuum Hypothesis
(CH), for each (non-compact) locally compact space $\Omega$ and
each non-modular prime ideal $P$ in $\C_0(\Omega)$ with
$\cardinal{\C_0(\Omega)/P}=\continuum$, there exists a radical
homomorphism from $\C_0(\Omega)$ with kernel precisely equal to
$P$. (For more details see \cite{dales2000}.)

Preceding this resolution was Bade and Curtis's theorem
\cite{badecurtis1960} which shows that each discontinuous
homomorphism from $\C_0(\Omega)$ into a Banach algebra $B$ can be
decomposed into a sum of a continuous homomorphism and a finite
number of discontinuous linear maps, each of which is a
homomorphism into the radical of $B$ when restricted to a maximal
ideal of $\C_0(\Omega)$. The following statement of the theorem also
includes some improvements from \cite{esterle1978} and
\cite{sinclair1975} (see also \cite{dales2000},
\cite{sinclair1976}); see \S\ref{definition} for notations.
\begin{theorem}
    \label{fip_bcse}
    Let $\Omega$ be a locally compact space, and let $\theta$ be a discontinuous homomorphism
    from $\C_0(\Omega)$ into a Banach algebra $B$. Suppose that
    $\theta(\C_0(\Omega))$ is dense in $B$.
    \begin{enumerate}
        \item The continuity ideal $\cont(\theta)$ is the largest ideal of $\C_0(\Omega)$ on which $\theta$ is continuous.
        \item There exists a non-empty finite subset $\set{p_1,\ldots,p_n}$ of
        $\onepointcompactification{\Omega}$ such that
        \[
            \bigcap_{i=1}^n J_{p_i}\subset\cont(\theta)\subset\bigcap_{i=1}^n M_{p_i}.
        \]
        \item There exists a continuous homomorphism
        $\mu:\C_0(\Omega)\to B$ such that
        \[
            B=\mu(\C_0(\Omega))\oplus \rad B,\quad \mu(\bigcap_{i=1}^nM_{p_i})\cdot\rad B=\set{0},
        \]
        and $\mu=\theta$ on a dense subalgebra of $\C_0(\Omega)$ containing $\cont(\theta)$.
        \item Set $\nu=\theta-\mu$. Then $\nu$ maps into $\rad B$, and the restriction of
        $\nu$ to $\bigcap_{i=1}^n M_{p_i}$ is a homomorphism $\nu^\prime$ onto a dense subalgebra of $\rad B$ such that
        $\cont(\theta)=\ker\nu^\prime$.
        \item There exist linear maps $\nu_1,\ldots, \nu_n:\C_0(\Omega)\to \rad B$ such
        that
        \begin{enumerate}
            \item $\nu=\nu_1+\cdots+\nu_n$,
            \item each $\nu_i|M_{p_i}$ $(1\le i\le n)$ is a non-zero radical homomorphism, and
            \item $\nu_i(\C_0(\Omega))\cdot\nu_j(\C_0(\Omega))=\set{0}$ for each $1\le i\neq j\le n$.
        \end{enumerate}
        \item The ideals $\ker\theta$ and $\cont(\theta)$ are always intersections of prime ideals.\enproof
    \end{enumerate}
\end{theorem}

In particular, this result emphasizes the roles of prime ideals
and of radical homomorphisms as building
blocks for general discontinuous homomorphisms from
$\C_0(\Omega)$: If we know about the radical homomorphisms from
$\C_0(\Omega)$ we will know about the homomorphisms from
$\C_0(\Omega)$ into Banach algebras. Dales and Esterle's theorem
shows how to construct radical homomorphisms from $\C_0(\Omega)$
with kernel being finite intersection of (non-modular) prime
ideals. In fact, for some spaces $\Omega$, the kernels of radical
homomorphisms from $\C_0(\Omega)$ are always finite intersections
of prime ideals (\cite{esterle1978}, \cite{pham2005b}).

However, in \cite{pham2005b}, we show that for most metrizable
non-compact locally compact spaces $\Omega$, for example $\reals$,
there exists a radical homomorphism from $\C_0(\Omega)$ whose
kernel is not the intersection of any finite number of prime
ideals.

In this paper, we shall show that the kernel of a radical
homomorphism from $\C_0(\Omega)$ is always the intersection of a
relatively compact family of (non-modular) prime ideals. We also
prove that, assuming the Continuum Hypothesis (CH), under a minor
cardinality condition, when restricted to those ideals that are
intersections of countably many prime ideals, the kernels of
radical homomorphisms from $\C_0(\Omega)$ are exactly the
intersections of relatively compact family of non-modular prime
ideals in $\C_0(\Omega)$. Similar result holds for continuity
ideals of homomorphisms from $\C_0(\Omega)$ into Banach algebras.
(See \S\ref{main_results}.)

\begin{remark} The
Continuum Hypothesis is required in construction of discontinuous
homomorphisms from $\C_0(\Omega)$ into Banach algebras, for it has
been proved by Solovay and Woodin that it is relatively consistent
with ZFC that all such homomorphisms are continuous (see
\cite{daleswoodin1987} for the proof and references).
\end{remark}

\section{Preliminary definitions and notations}
\label{definition}

Let $A$ be a commutative algebra. The (\emph{conditional})
\emph{unitization} $\unitA$ of $A$ is defined as the algebra $A$
itself if $A$ is unital, and as $A$ with identity adjoined
otherwise. The identity of $\unitA$ is denoted by $\unit_A$.

A prime ideal or semiprime ideal in $A$ must be a proper ideal.
However, we consider $A$ itself as a finite intersection of prime
ideals (\emph{the intersection of the empty collection of prime
ideals}).

Define the \emph{prime radical} $\primeradical{I}$ of an ideal $I$
in $A$ to be the intersection of all prime ideals in $A$
containing $I$, so that
\[
    \primeradical{I}=\set{a\in A:\ a^n\in I \ \textrm{for some}\
n\in\naturals}\,.
\]

For each ideal $I$ in $A$ and each element $a\in \unitA$, define
the \emph{quotient of $I$ by $a$} to be the ideal
\[
    \quotoel{I}{a}\,=\set{b\in A:\ ab\in I}\,.
\]
Clearly we have $I \subset\quotoel{I}{a}$ in each case.

Let $I$ be an ideal in $\unitA$. A subset $S$ of $\unitA$ is
\emph{algebraically independent modulo} $I$ if
$p(a_1,\ldots,a_n)\notin I$ for each $n\in\naturals$, each
non-zero polynomial $p\in\complexs[X_1,\ldots,X_n]$, and each
$n$-tuple $(a_1,\ldots,a_n)$ of distinct elements of $S$. A
\emph{transcendence basis} for $\unitA$ modulo $I$ is a maximal
set among all the subsets of $\unitA$ which are algebraically
independent modulo $I$; such a basis always exists.

For an well-ordered set $\Lambda$, we denote by $\order(\Lambda)$
the ordinal that is order isomorphic to $\Lambda$.

For the definition of \emph{universal} algebras, see
\cite[Definition 5.7.8]{dales2000}. The important fact that we
need is the existence of universal radical Banach algebras. For
example, the integral domain $\Lone(\reals^+,\omega)$ is universal
for each radical weight $\omega$ bounded near the origin
\cite[Theorem 5.7.25]{dales2000}. Indeed, the class of universal
commutative radical Banach algebras has been characterized in
\cite{esterle1981} (see also \cite[Theorem 5.7.28]{dales2000}).

Let $B$ be a Banach algebra, and let $S$ be an indexing set.
Define $\linfty(S,B)$ to be the Banach algebra of all bounded
families $(b_\alpha:\alpha\in S)$ in $B$ under pointwise algebraic
operations and the supremum norm.

For a discussion of the theory of the algebras of continuous
functions, see \cite{dales2000}, \cite{daleswoodin1996} or
\cite{gillmanjerison1960}. Here, we just give some facts which are
needed in our discussion.

Let $\Omega$ be a locally compact space; the convention is that
locally compact spaces and compact spaces are Hausdorff. The
\emph{one-point compactification} of $\Omega$ is denoted by
$\onepointcompactification{\Omega}$. Denote by $\C_c(\Omega)$ the
\emph{algebra of compactly supported continuous functions} on
$\Omega$. For each $p\in \Omega$, define
\[
\begin{array}{rcl}
J_p &=& \set{f\in\C_0(\Omega):\ f\ \textrm{is zero on a neighbourhood of}\ p},\\
M_p &=& \set{f\in\C_0(\Omega):\ f(p)=0}.
\end{array}
\]
For $p$ being the point (at infinity) adjoined to $\Omega$ to
obtain $\onepointcompactification{\Omega}$, we also set
\[
J_p =\C_c(\Omega)\quad\textrm{and}\quad M_p=\C_0(\Omega).
\]

For each prime ideal $P$ in $\C_0(\Omega)$, there always exists a
unique point $p\in\onepointcompactification{\Omega}$ such that
$J_p\subset P\subset M_p$, we say that $P$ \emph{is supported at
the point} $p$. It can be seen that $P$ is modular if and only if
its support point belongs to $\Omega$.

It is an important fact that, for each prime ideal $P$ in
$\C_0(\Omega)$, the set of prime ideals in $\C_0(\Omega)$ which
contain $P$ is a chain with respect to the inclusion relation.

For each function $f$ continuous on $\Omega$, the \emph{zero set}
of $f$ is denoted by $\zero(f)$. The set of zero sets of
continuous functions on $\Omega$ is denoted by $\zero[\Omega]$.

A \emph{$z$-filter} $\F$ on $\Omega$ is a non-empty proper subset
of $\zero[\Omega]$ satisfying:
\begin{enumerate}
    \item $Z_1\cap Z_2$ belongs to $\F$ whenever both $Z_1$ and $Z_2$ belong to
    $\F$,
    \item if $Z_1\in\F$, $Z_2\in\zero[\Omega]$ and $Z_1\subset Z_2$, then $Z_2$ also belongs to
    $\F$.
\end{enumerate}
Each $z$-filter $\F$ corresponds to an ideal
\[
    \set{f\in\C(\Omega):\ \zero(f)\in\F},
\]
denoted by $\zero^{-1}[\F]$; each such ideal is called a
\emph{$z$-ideal}.

\section{Relatively compact families of prime ideals}

In this section, let $A$ be a commutative algebra.

\begin{definition}[cf. \cite{pham2005b} Definition 3.1]
    \label{pseudo_finiteness_definition}
An indexed family $(P_i)_{i\in S}$ of prime ideals in $A$ is
\emph{pseudo-finite} if $a\in P_i$ for all but finitely many $i\in
S$ whenever $a\in \bigcup_{i\in S} P_i$.
\end{definition}

For a pseudo-finite sequence $(P_n)$ of prime ideals, it is
obvious that $\bigcup_{n=1}^\infty P_n$ is either a prime ideal in
$A$ or the whole $A$.

\begin{definition}
    \label{compactness_definition}
    A family $\mathfrak{C}$ of prime ideals in $A$ is \emph{relatively compact}
    if every sequence of prime ideals in $\mathfrak{C}$ contains a pseudo-finite
    subsequence. The family $\mathfrak{C}$ is \emph{compact} if it is
    relatively compact and contains the union of each of its pseudo-finite sequence.
\end{definition}

Obviously, the union of finitely many pseudo-finite families is
relatively compact. In the rest of this section, we shall justify
our choice of terminology.

Denote by $\Pi$ the set of prime ideals in $A$. For
$a_1,a_2,\ldots, a_m$, and $b$ in $A$, define
\[
    \U_{a_1,\ldots,a_m}^b=\set{P\in\Pi:\ a_i\in P\ (1\le i\le m)\ \textrm{and}\ b\notin
P}.
\]
Then the collection of all $\U_{a_1,\ldots,a_m}^b$'s is a base for
a topology $\tau$. Indeed, by the primeness, we have
\[
    \U_{a_1,\ldots,a_m,c_1,\ldots,c_n}^{bd}=\U_{a_1,\ldots,a_m}^b\cap\U_{c_1,\ldots,c_n}^d.
\]
Moreover, we \emph{claim} that $\tau$ is Hausdorff. For, let
$P_1\neq P_2\in\Pi$. If $P_1\nsubset P_2$ and $P_2\nsubset P_1$,
say $a\in P_1\setminus P_2$ and $b\in P_2\setminus P_1$, then
$P_1\in \U_a^b$, $P_2\in \U_b^a$, and $\U_a^b\cap
\U_b^a=\emptyset$. Otherwise, say $P_1\subsetneq P_2$, choose
$b\in P_2\setminus P_1$ and $c\in A\setminus P_2$, then $P_1\in
\U_0^b$, $P_2\in \U_b^c$, and $\U_0^b\cap \U_b^c=\emptyset$.

Next, we \emph{claim} that $\U_0^u$ is $\tau$-compact ($u\in A$).
Indeed, we see that $\set{\U_a^u,\ \U_0^a:\ a\in A}$ is a subbasis
for the relative $\tau$-topology on $\U_0^u$, so by Alexander's
lemma, we need only to show that each cover of $\U_0^u$ by sets in
this subbasis has a finite subcover. So, let $E, F$ be subsets of
$A$ such that
\[
    \U_0^u=\bigcup_{a\in E}\U_a^u\cup\bigcup_{b\in F}\U_0^{b}.
\]
Set $S=\set{u^ma_1\cdots a_n:\ m,n\in\naturals,\ a_1,\ldots,a_n\in
E}$, and let $I$ be the ideal generated by $F$. Assume toward a
contradiction that $S\cap I=\emptyset$. Then since $S$ is closed
under multiplication, there exists a prime ideal $P$ such that
$P\supset I$ and $P\cap S=\emptyset$; this implies that $P\in
\U_0^u$ but $P\notin \bigcup_{a\in E}\U_a^u\cup\bigcup_{b\in
F}\U_0^{b}$, a contradiction. Thus, $S\cap I\neq\emptyset$, so
there exist $k,m,n\in\naturals$, $a_1,\ldots,a_m\in E$ and
$b_1,\ldots,b_n\in F$, and $c_1,\ldots, c_n\in A$ such that
$u^ka_1\cdots a_m=b_1c_1+\cdots b_nc_n$. We can then deduce that
\[
    \U_0^u=\bigcup_{i=1}^m\U_{a_i}^u\cup\bigcup_{j=1}^n\U_0^{b_j}.
\]
Thus $\tau$ is locally compact.

The one point compactification of $(\Pi,\tau)$ can be considered
as the set $\Pi\cup\set{A}$; a basis of neighbourhood for $A$ is
given by
\[
    \U_{b_1,\ldots,b_n}=\set{A}\cup\set{P\in\Pi:\ b_i\in P\ (1\le i\le n)}.
\]

\begin{proposition}
    Let $A$ be a commutative algebra. Denote by $\Pi$ the set of prime
    ideals in $A$. Define a topology $\tau$ as above.
    Then $(\Pi,\tau)$ is a totally disconnected locally compact
    space, and every [relative] compact family of prime ideals in
    $A$ is a [relatively] sequentially $\tau$-compact subset of
    $\Pi\cup\set{A}$.
\end{proposition}
\begin{proof}
It remains to prove the last assertion. We \emph{claim} that a
pseudo-finite sequence $(P_n)$ of prime ideals in $A$ is
$\tau$-convergent in $\Pi\cup\set{A}$. In fact, set
$P=\bigcup_{n=1}^\infty P_n$. Then either $P\in \Pi$ or $P=A$. In
both cases, we can check that $(P_n)$ $\tau$-converges to $P$.
\end{proof}

\begin{remark}
Suppose that $A$ is unital. Then $\Pi=\U_0^{\unit_A}$, and so
$(\Pi,\tau)$ is compact. In the above proposition, we can replace
$(\Pi\cup\set{A},\tau)$ by $(\Pi,\tau)$.
\end{remark}

\begin{remark}
Suppose that $A=\C_0(\Omega)$ for some locally compact space
$\Omega$. Then for each pseudo-finite sequence $(P_n)$ of prime
ideals in $\C_0(\Omega)$, the union $\bigcup_{n=1}^\infty P_n$ is
in fact a prime ideal in $\C_0(\Omega)$. Again, in the above
proposition, we can replace $(\Pi\cup\set{A},\tau)$ by
$(\Pi,\tau)$.
\end{remark}

\begin{remark}
Let us instead consider a topology $\sigma$ on $\Pi\cup\set{A}$
generated by
\[
    \U_{a_1,\ldots,a_m}^Q=\set{P\in\Pi\cup\set{A}:\ P\subset Q\ \textrm{and}\ a_i\in P\ (1\le i\le
m)},
\]
where $Q$ is either a semiprime ideal in $A$ or $A$ itself, and
$a_1,\ldots, a_m\in Q$. Then a sequence of prime ideals in $A$ is
pseudo-finite if and only if it is convergent in
$(\Pi\cup\set{A},\sigma)$, and so a family of prime ideals in $A$
is [relatively] compact if and only if it is [relatively]
sequentially compact in $(\Pi\cup\set{A},\sigma)$. However, in
general, $\sigma$ is neither Hausdorff nor locally compact. In the
case where $A=\C_0(\Omega)$, then $(\Pi\cup\set{A},\sigma)$ is
Hausdorff (but not locally compact).
\end{remark}

\section{Abstract continuity ideals and relatively compact families of prime
ideals} \label{abstract_continuity_ideals}

Let $\theta :A\to B$ be a homomorphism from a commutative Banach
algebra $A$ into a Banach algebra $B$. Let $(a_n:\ n\in\naturals)$
be a sequence in $A$. Then
\[
    \quotoel{\cont(\theta)}{a_1a_2\cdots a_n}\subset
    \quotoel{\cont(\theta)}{a_1a_2\cdots a_{n+1}}\qquad (n\in\naturals).
\]
It follows easily from the stability lemma (see
\cite[5.2.7]{dales2000} or \cite[1.6]{sinclair1976} for the
statement and proof) that there exists $n_0$ such that
\[
    \quotoel{\cont(\theta)}{a_1a_2\cdots a_n}=
    \quotoel{\cont(\theta)}{a_1a_2\cdots a_{n+1}}\qquad (n\ge n_0).
\]
Thus $\cont(\theta)$ is an \emph{abstract continuity ideal} in the
following sense.

\begin{definition}
    Let $A$ be a commutative algebra. An ideal $I$ is an
    \emph{abstract continuity ideal} if, for each sequence $(a_n)$ in $A$, there exists $n_0$ such that
    \[
        \quotoel{I}{a_1a_2\cdots a_n}=
        \quotoel{I}{a_1a_2\cdots a_{n+1}}\qquad (n\ge n_0).
    \]
\end{definition}

\begin{proposition}
    \label{abstract_first_part}
    Let $\setofprimes$ be a relatively compact family of
    prime ideals in a commutative algebra $A$. Then $\bigcap\set{P:P\in\setofprimes}$ is
    an abstract continuity ideal in $A$.
\end{proposition}
\begin{proof}
    Denote by $I$ the intersection
    $\bigcap\set{P:P\in\setofprimes}$. Assume toward a
    contradiction that $I$ is not an abstract continuity ideal.
    Then there exists a sequence $(a_n)$ in $A$ such that
    \[
        \quotoel{I}{a_1a_2\cdots a_n}\subsetneq
        \quotoel{I}{a_1a_2\cdots a_{n+1}}\qquad (n\in\naturals).
    \]
    For each $n$, we see that
    \[
        \quotoel{I}{a_1\cdots a_n}=\bigcap\set{P\in\setofprimes:\
        a_1\cdots a_n\notin P}.
    \]
    Thus, it follows that, there exists $P_n\in\setofprimes$ such that
    $a_1\cdots a_n\notin P_n$ but $a_1\cdots a_{n+1}\in P_n$. The relative compactness implies
    that there exists $n_1<n_2<\cdots$ such that  $(P_{n_i})$ is pseudo-finite.
    However, we see that $a_1\cdots a_{n_2}\in P_{n_1}$, but $a_1\cdots
    a_{n_2}\notin P_{n_i}$ ($i\ge 2$); this contradicts the pseudo-finiteness.
\end{proof}

The remaining of this section is devoted to a converse of the
above proposition. \emph{Let $I$ be an abstract continuity ideal
of a commutative algebra $A$. Denote by $\setofprimes$ the set of
prime ideals of the form $\quotoel{I}{a}$ for some $a\in A$.}  The
following is a modification of \cite[4.3]{pham2005b}.

\begin{lemma}
    \label{dichotomy}
    For each cardinal $\kappa\le\cardinal{\setofprimes}$,
         there exists a sub-family $\G\subset \setofprimes$ with the properties that
         $\cardinal{\G}\ge\kappa$ and that $\cardinal{\set{P\in\G:\ a\notin P}}<\kappa$
         for each $a\in\bigcup\set{P:\ P\in\G}$.
\end{lemma}
\begin{proof}
For each $a\in A\cup\set{\unit_A}$, let $\setofprimes_a$ be
the set of prime ideals of the form $\quotoel{I}{ab}$ for some
$b\in A$. We \emph{claim} that there exists $a_0\in
A\cup\set{\unit_A}$ such that $\abs{\setofprimes_{a_0}}\ge\kappa$
and such that, for each $a\in A$, either
$\abs{\setofprimes_{a_0a}}<\kappa$ or
$\quotoel{I}{a_0a}=\quotoel{I}{a_0}$. Indeed, assume the contrary.
Then, since $\abs{\setofprimes_{\unit_A}}\ge\kappa$, by induction,
there exists a sequence $(a_n)\subset A$ such that
$\abs{\setofprimes_{a_1\cdots a_n}}\ge\kappa$ and such that
\[
    \quotoel{I}{a_1\cdots a_n}\subsetneq
    \quotoel{I}{a_1\cdots a_{n+1}} \qquad (n\in\naturals).
\]
This contradicts the definition of an continuity ideal. Hence the
claim holds.

Put $\G=\setofprimes_{a_0}$; this obviously satisfies
$\cardinal{\G}\ge\kappa$. Suppose that $a\in A$ and that
$\G'=\set{P\in \G:\, a\notin P}$ has cardinality at least
$\kappa$. Then, for each $P\in \G'$, because $a\notin P$ we have
$\quotoel{P}{a}=P$. Thus $\G'\subset\setofprimes_{a_0a}\,$, and
hence $\abs{\setofprimes_{a_0a}}\ge\kappa$. Therefore, by the
claim, we must have $\quotoel{I}{a_0a}=\quotoel{I}{a_0}$. We now
show that $\G'=\G$. Assume towards a contradiction that $\G'\neq
\G$, and let $P\in \G\setminus\G'$, say $P=\quotoel{I}{a_0a_1}$
for some $a_1\in A$. Then, since $a\in P$ we have $a_1\in
\quotoel{I}{a_0a}=\quotoel{I}{a_0}$, so that $a_0a_1\in I$. This
implies that $P=A$, a contradiction. This proves that $\G$ has the
desired property.
\end{proof}

\begin{lemma}
    \label{prime_radical_intersection}
    $\primeradical{I}$ is the intersection of the prime ideals in $\setofprimes$.
\end{lemma}
\begin{proof}
This is based on the commutative prime kernel theorem due to
Sinclair (see \cite[Theorem 5.3.15]{dales2000} or \cite[Theorem
11.4]{sinclair1976}). The proof of \cite[Lemma 4.1]{pham2005b}
works almost verbatim.
\end{proof}

\begin{lemma}
    \label{no_descending}
    Every element in $\setofprimes$ contains a minimal element.
\end{lemma}
\begin{proof}
    Assume toward a contradiction that there exists $(P_n=\quotoel{I}{f_n})\subset\setofprimes$
    such that
    \[
        P_1\supsetneq P_2\supsetneq\cdots\supsetneq P_n\supsetneq\cdots.
    \]
    For each $n$, choose $a_n\in A$ such that $a_n\in P_n\setminus P_{n+1}$. Then
    we see that $a_1\cdots a_nf_{n}\in I$ but $a_1\cdots
    a_{n}f_{n+1}\notin I$. Thus
    \[
        \quotoel{I}{a_1a_2\cdots a_n}\subsetneq
        \quotoel{I}{a_1a_2\cdots a_{n+1}}\qquad (n\in\naturals);
    \]
    a contradiction to $I$ being an abstract continuity ideal.
\end{proof}

\begin{lemma}
    \label{non_redundancy_continuity_ideals}
    Let $P$ be in $\setofprimes$. Then there
    exists $a\notin P$ but $a\in Q$ for all
    $Q\in\setofprimes$ such that $Q\nsubset P$.
\end{lemma}
\begin{proof}
Assume the contrary. Pick $a_1\notin P$. Suppose that we have
already picked $a_1,\ldots,a_n\notin P$. By the assumption, we can
find $Q_n\in\setofprimes$ such that $a_1\ldots a_n\notin Q_n$ and
$Q_n\nsubset P$. We can then choose $a_{n+1}\in Q_n\setminus P$.
The induction can be continued. We see that $(a_n)$, $(Q_n)$
constructed satisfy $a_1\ldots a_n\notin Q_n$ but $a_1\ldots
a_{n+1}\in Q_n$ ($n\in\naturals$). Let $Q_n=\quotoel{I}{f_n}$.
Then we see that $f_{n}\in\quotoel{I}{a_1a_2\cdots
a_{n+1}}\setminus \quotoel{I}{a_1a_2\cdots a_{n}}$; a
contradiction to $I$ being an abstract continuity ideal.
\end{proof}

\begin{lemma}
    \label{subcollection}
    Let $\set{P_\alpha:\ \alpha\in S}$ be a subfamily of
    $\setofprimes$. Then $\bigcap_{\alpha\in S} P_\alpha$ is also
    an abstract continuity ideal.
\end{lemma}
\begin{proof}
    Assume the contrary. Then there exists $(a_n)$ such that
    \[
        \quotoel{\left(\bigcap_{\alpha\in S} P_\alpha\right)}{a_1a_2\cdots a_n}\subsetneq
        \quotoel{\left(\bigcap_{\alpha\in S} P_\alpha\right)}{a_1a_2\cdots a_{n+1}}\qquad (n\in\naturals).
    \]
    For each $n$, choose $b_n\in A$ such that $a_1\cdots
    a_nb_n\notin \bigcap_{\alpha\in S} P_\alpha$ but $a_1\cdots
    a_{n+1}b_n\in \bigcap_{\alpha\in S} P_\alpha$. Then choose
    $\alpha_n\in S$ such that $a_1\cdots a_nb_n\notin
    P_{\alpha_n}$. We have $P_{\alpha_n}=\quotoel{I}{f_{\alpha_n}}$ for
    some $f_{\alpha_n}\in A$. We see that $a_1\cdots
    a_nb_nf_{\alpha_n}\notin I$ but $a_1\cdots
    a_{n+1}b_nf_{\alpha_n}\in I$. Thus
    \[
        \quotoel{I}{a_1a_2\cdots a_n}\subsetneq
        \quotoel{I}{a_1a_2\cdots a_{n+1}}\qquad (n\in\naturals);
    \]
    a contradiction to $I$ being an abstract continuity ideal.
\end{proof}

\begin{lemma}
  \label{equivalent_primes}
 Let $J$ be a semiprime ideal in $A$. Let $a,b \in A$ be
 such that $\quotoel{J}{a}$ and $\quotoel{J}{b}$ are prime ideals.
 Then the following are equivalent:
  \begin{itemize}
    \item[\rm{(a)}] $\quotoel{J}{a}\subset\quotoel{J}{b}$,
    \item[\rm{(b)}] $ab\notin J$,
    \item[\rm{(c)}] $\quotoel{J}{a}=\quotoel{J}{b}$.
  \end{itemize}
\end{lemma}
\begin{proof}

(a)$\Rightarrow$(b): Since $J$ is semiprime and $\quotoel{J}{b}$
is a proper ideal in $A$, we see that $b\notin \quotoel{J}{b}$. So
$b\notin \quotoel{J}{a}$, and therefore $ab\notin J$.

(b)$\Rightarrow$(c): Condition (b) implies that $a\notin
\quotoel{J}{b}$. Let $f\in\quotoel{J}{a}$. Then $fa\in
J\subset\quotoel{J}{b}$, and so, by the primeness of
$\quotoel{J}{b}$, $f\in\quotoel{J}{b}$. Thus
$\quotoel{J}{a}\subset\quotoel{J}{b}$. Similarly, we have
$\quotoel{J}{b}\subset\quotoel{J}{a}$.
\end{proof}

\begin{remark}
In the case where $I$ is semiprime, the above lemma shows that,
for each $P=\quotoel{I}{a}\in\setofprimes$, $P$ is minimal in
$\setofprimes$ and $a\notin P$ but $a\in Q$ whenever
$Q\in\setofprimes\setminus\set{P}$.
\end{remark}

We can now state the main result of this section.

\begin{theorem}
    Let $I$ be an abstract continuity ideal of a commutative algebra $A$.
    Denote by $\setofprimes_0$ the set of minimal ideals among the prime ideals of the form
    $\quotoel{I}{a}$ for some $a\in A$.
    Then:
    \begin{enumerate}
         \item $\primeradical{I}=\bigcap\set{P: P\in\setofprimes_0}$;
         \item $\setofprimes_0$ is a relatively compact family of prime ideals.
    \end{enumerate}
\end{theorem}
\begin{proof}
The first assertion follows from Lemmas
\ref{prime_radical_intersection} and \ref{no_descending}.

For the second one, let $(P_n)\subset \setofprimes_0$. Set
$J=\bigcap_{n=1}^\infty P_n$. By Lemma
\ref{non_redundancy_continuity_ideals}, there exists
$a_n\in\bigcap_{i\neq n} P_i\setminus P_n$, and so
$P_n=\quotoel{J}{a_n}$. Let $a\in A$ be such that $\quotoel{J}{a}$
is a prime ideal. We \emph{claim} that
$\quotoel{J}{a}\in\set{P_n}$. Indeed, we see that $a\notin J$, and
thus $a\notin P_{n_0}$ for some $n_0$. So, $aa_{n_0}\notin J$. By
Lemma \ref{equivalent_primes}, we deduce that
$\quotoel{J}{a}=P_{n_0}$.

It then follows from Lemmas \ref{subcollection} and
\ref{dichotomy} (applied to $J$) that $(P_n)$ must have a
pseudo-finite subsequence.
\end{proof}

\begin{corollary}
    Let $\theta:A\to B$ be a homomorphism from a commutative
    Banach algebra $A$ into a Banach algebra $B$. Then $\primeradical{\cont(\theta)}$
    is the intersection of a relatively compact family of prime ideals
    of the form $\quotoel{\cont(\theta)}{a}$ for $a\in A$. \enproof
\end{corollary}

\begin{lemma}
    Let $I$ be an abstract continuity ideal of $\C_0(\Omega)$ for a locally compact space $\Omega$.
    Then $I$ is either a semiprime ideal or the whole of $\C_0(\Omega)$.
\end{lemma}
\begin{proof}
    The proof is the same as the proof that the continuity ideal of a discontinuous homomorphism
    from $\C_0(\Omega)$ into a Banach algebra is semiprime (\cite{esterle1978},\cite{sinclair1975},
    cf. \cite[Theorem 5.4.31]{dales2000}).
\end{proof}

\begin{corollary}
    \label{main_theorem_abstract}
    Let $\Omega$ be a locally compact space.
    \begin{enumerate}
        \item Let $I$ be an abstract continuity ideal in $\C_0(\Omega)$.
            Denote by $\setofprimes$ the set of prime ideals
            of the form $\quotoel{I}{f}$ for some
            $f\in\C_0(\Omega)$.
            Then:
            \begin{enumerate}
                \item $I=\bigcap\set{P: P\in\setofprimes}$;
                \item $\setofprimes$ is a relatively compact family of prime ideals.
            \end{enumerate}
        \item Conversely, let $\setofprimes$ be a relatively compact family of
            prime ideals in $\C_0(\Omega)$. Then $\bigcap\set{P:P\in\setofprimes}$ is
            an abstract continuity ideal in
            $\C_0(\Omega)$. \enproof
    \end{enumerate}
\end{corollary}

\begin{corollary}
    Let $\Omega$ be a locally compact space. Then each homomorphism from
    $\C_0(\Omega)$ into a Banach algebra is continuous on the intersection
    of a relatively compact family of prime ideals of the form
    $\quotoel{\cont(\theta)}{f}$ for $f\in \C_0(\Omega)$. \enproof
\end{corollary}

\section{More properties of relatively compact families of prime
ideals}

\label{moreproperties}

 In this section, let $\Omega$ be a locally compact space, and let $\setofprimes$
 be a non-empty relatively compact family of prime ideals in $\C_0(\Omega)$.  Denote by
 $\setofunions$ the collection of all the ideals that are unions of \emph{countably} many ideals in
$\setofprimes$. We call $\setofunions$ the \emph{closure} of
$\setofprimes$; we shall show that it is indeed the smallest
compact family of prime ideals containing $\setofprimes$.

\emph{Note that} an ideal in $\setofunions$ is automatically prime
in $\C_0(\Omega)$, and that the union of each pseudo-finite
sequence of prime ideals in $\C_0(\Omega)$ is again a prime ideal
in $\C_0(\Omega)$ (see the next lemma).

\begin{lemma}
    The union of finitely many prime ideals in $\C_0(\Omega)$ is either one of the given prime ideal
    or not an ideal. The union of countably many prime ideals in $\C_0(\Omega)$ is
    not equal $\C_0(\Omega)$.
\end{lemma}

\begin{proof}
We prove the second clause only; the proof of the first one is
similar. Let $P_n$ ($n\in\naturals$) be prime ideals in
$\C_0(\Omega)$. Choose $f_n\in \C_0(\Omega)\setminus P_n$. We can
assume that $0\le f_n\le 2^{-n}$. Set $f=\sum_{n=1}^\infty f_n$.
Then $f\in\C_0(\Omega)$ but $f\notin P_n$ ($n\in\naturals$) since
$f\ge f_n$.
\end{proof}

\begin{lemma}
    Each chain in $\setofunions$ is well-ordered with respect to
    the inclusion; that is, each non-empty chain in $\setofunions$ has a smallest element.
\end{lemma}

\begin{proof}
Assume the contrary, then we can find an infinite chain
$\cdots\subsetneq Q_n\subsetneq \cdots \subsetneq Q_1$ in
$\setofunions$. For each $n$, choose $P_n\in\setofprimes$ such
that $P_n\subset Q_n$ but $P_n\nsubset Q_{n+1}$. By the relative
compactness of $\setofprimes$ and without loss of generality, we
can assume that $(P_n:n\in\naturals)$ is a pseudo-finite sequence.
Set $Q=\bigcup_{n=1}^\infty P_n$. Then $Q\in\setofunions$, and for
each $n\in\naturals$, either $Q_n\subset Q$ or $Q\subset Q_n$.
Since $P_{n-1}\nsubset Q_{n}$, we must have $Q_n\subset Q$ ($n\ge
2$). Choose $a\in Q_2\setminus Q_3$. Then $a\notin Q_n$, and so
$a\notin P_n$ ($n\ge 3$). However, $a\in Q=\bigcup_{n=1}^\infty
P_n$. This contradicts the pseudo-finiteness of $(P_n)$.
\end{proof}

\begin{lemma}
    $\setofunions$ is compact.
\end{lemma}
\begin{proof}
    Let $(Q_n)$ be a sequence in $\setofunions$. Let
    $P_n\in\setofprimes$ such that $P_n\subset Q_n$. We can find a
    pseudo-finite subsequence $(P_{n_i})$; the union of which is denoted by
    $Q$. We have either $Q_{n_i}\subset Q$ or $Q\subset Q_{n_i}$.
    If there are infinitely many $Q_{n_i}$ contained in $Q$, then
    those $Q_{n_i}$ form a pseudo-finite sequence. On the other
    hand, if there are infinitely many $Q_{n_i}$ containing $Q$,
    then those $Q_{n_i}$ form a chain, and the previous lemma
    enable us to find an increasing sequence of ideals. Thus
    $\setofunions$ is relatively compact. The result then follows
    from the definition of $\setofunions$.
\end{proof}

\begin{lemma}
    $\setofunions$ is the set of unions of pseudo-finite sequences
    of ideals in $\setofprimes$.
\end{lemma}
\begin{proof}
 We need only to prove that each ideal $Q\in\setofunions$ is the
 union of a pseudo-finite sequence in $\setofprimes$. For this
 purpose, we only need to consider the case where $\setofprimes$
 is countable and that $Q$ is the largest ideal in $\setofunions$.
 It is easy to see that, in this case, any chain in
 $\setofunions$ is countable.

\emph{Case 1: $Q$ is the union of a chain of ideals in
 $\setofunions\setminus\set{Q}$}. By the countability
 and well-ordering of the chain, there exist $Q_1\subsetneq Q_2\subsetneq\cdots\subsetneq
 Q$ such that $Q=\bigcup_{n=1}^\infty Q_n$. For each $n$, choose
 $P_n\in\setofprimes$ such that $P_n\subset Q_{n+1}$ but $P_n\nsubset Q_n$. By the relative
compactness of $\setofprimes$, we can choose a pseudo-finite
subsequence $(P_{n_i})$ with union $Q'$. Then $Q'\subset Q$, and
for each $i\ge 2$, either $Q_{n_i}\subset Q'$ or $Q'\subset
Q_{n_i}$. Since $P_{n_i}\nsubset Q_{n_i}$, we must have
$Q_{n_i}\subset Q'$ ($i\ge 2$). Thus $Q=Q'$.

\emph{Case 2: $Q$ is not the union of any chain of ideals in
 $\setofunions\setminus\set{Q}$}. Then any $P\in\setofunions\setminus\set{Q}$
 is contained in a maximal element of
 $\setofunions\setminus\set{Q}$. Since $Q$ cannot be the union of
 any finite number of prime ideals properly contained in $Q$, either
$\setofunions=\set{Q}$ which implies that $Q\in\setofprimes$ or
 there exists infinitely many maximal elements of $\setofunions\setminus\set{Q}$. In the latter case, let $Q_n$ ($n\in\naturals$) be distinct maximal elements of $\setofunions\setminus\set{Q}$.
 Choose $P_n\in\setofprimes$ such that $P_n\subset Q_n$. By the relative
compactness of $\setofprimes$ and without loss of generality, we
can assume that $(P_n:n\in\naturals)$ is a pseudo-finite sequence.
Set $Q'=\bigcup_{n=1}^\infty P_n$. Then $Q'\in\setofunions$,
$Q'\subset Q$, and for each $n\in\naturals$, either $Q_n\subset
Q'$ or $Q'\subset Q_n$. The maximality and distinction of
$Q_{n}$'s imply that $Q'\subset Q_n$ for at most one
$n\in\naturals$. The maximality of $Q_n$'s again implies that
$Q'=Q$.

In both cases, we see that $Q$ is the union of a pseudo-finite
sequence in $\setofprimes$.
\end{proof}

Summary; note that the property (iii) is indeed a consequence of
(ii):
\begin{proposition}
\label{compact}
    $\setofunions$ satisfies the following:
    \begin{enumerate}
        \item $\setofunions$ is the set of unions of pseudo-finite sequences
    of ideals in $\setofprimes$;
        \item $\setofunions$ is compact;
        \item each chain in $\setofunions$ is well-ordered;
        \item $\bigcap\setofprimes=\bigcap\setofunions$. \enproof
    \end{enumerate}
\end{proposition}

From (i), we see that $\setofunions$ is the smallest compact
family of prime ideals containing $\setofprimes$. Property (iii)
also shows that the intersection of $\setofprimes$ is equal the
intersection of its minimal elements.

\myskip

In the remaining of the section, we consider $\setofunions$ to be
any compact family of prime ideals in $\C_0(\Omega)$.

\begin{lemma}
    \label{nonredundancy}
    Let $P$ be in $\setofunions$. Then there
    exists $a\notin P$ but $a\in Q$ for all
    $Q\in\setofunions$ such that $Q\nsubset P$.
\end{lemma}
\begin{proof}
Assume the contrary. As in Lemma
\ref{non_redundancy_continuity_ideals}, we can construct
$(a_n)\subset A$, $(Q_n)\subset \setofunions$ satisfying
$a_1\ldots a_n\notin Q_n$ but $a_1\ldots a_{n+1}\in Q_n$
($n\in\naturals$). By the compactness, $(Q_n)$ has a pseudo-finite
subsequence $(Q_{n_i})$. However, $a_1\ldots a_{n_2}\in Q_{n_1}$
but $a_1\ldots a_{n_2}\notin Q_{n_i}$ ($i\ge 2$); a contradiction
to the pseudo-finiteness.
\end{proof}

We say that an ideal $Q$ is a \emph{roof} of $\setofunions$ if it
is the union of the ideals in a maximal chain in $\setofunions$. A
roof must be either a prime ideal in $\C_0(\Omega)$ or
$\C_0(\Omega)$ itself.

\begin{lemma}
    \label{finiteroofs}
    $\setofunions$ has only finitely many roofs. Also, there are only
    finite many maximal modular ideals in $\C_0(\Omega)$
    such that each of them contains an ideal in $\setofunions$.
\end{lemma}

\begin{proof}
    We shall prove that there can be only finitely many \emph{disjoint} maximal chains in
    $\setofunions\setminus\set{\C_0(\Omega)}$; the lemma then follows. Assume the contrary
    that $\mathfrak{C}_n$ ($n\in\naturals$) are disjoint maximal
    chains in $\setofunions\setminus\set{\C_0(\Omega)}$. Pick $Q_n\in\mathfrak{C}_n$. Without
    loss of generality, we can suppose that $(Q_n)$ is
    pseudo-finite; the union is denoted by $Q$. We see that
    $Q\in\setofunions\setminus\set{\C_0(\Omega)}$, and so
    $Q\in\mathfrak{C}_n$ ($n\in\naturals$), contradicting the
    disjointness of $\mathfrak{C}_n$'s.
\end{proof}

The following lemma and proposition are based on a suggestion of an
anonymous referee of an initial version of our previous paper.

\begin{lemma}
  \label{divisibility}
    Suppose that $\setofunions$ is a compact family of prime ideals in
    $\C_0(\Omega)$ with only one maximal element $Q$. Set
    $I=\bigcap\setofunions$. Let $a\in \C_0(\Omega)\setminus Q$
    and let $b\in Q$. Then there exists $s\in Q$ such that
    $as-b\in I$.
\end{lemma}

\begin{proof}
    It is standard that for each prime ideal $P\subset Q$ there
    exists $s\in Q$ such that $as-b\in P$.

    Assume toward a contradiction that for all $s\in
    Q$ we have $as-b\notin I$. Set $s_1=0$,
    $b_1=b-as_1=b$, and $\setofunions_1=\set{P\in\setofunions:\
    b_1\notin P}$. Suppose that we have already construct $s_n\in Q$,
    $b_n=b-as_n$, and $\setofunions_n=\set{P\in\setofunions:\
    b_n\notin P}$ such that
    \[
        \setofunions_n\subsetneq\cdots\subsetneq\setofunions_1.
    \]
    By the assumption, we have $\setofunions_n\neq\emptyset$.
    Choose $P_0\in\setofunions_n$.
    Since $\displaystyle{\frac{b_n}{\sqrt{\abs{b_n}}}}\in Q$.
    There exists $s'\in Q$ such that $as'-\displaystyle{\frac{b_n}{\sqrt{\abs{b_n}}}}\in
    P_0$. Set $s_{n+1}=s_n+s'\sqrt{\abs{b_n}}$,
    \[
        b_{n+1}=b-as_{n+1}=\left(\frac{b_n}{\sqrt{\abs{b_n}}}-as'\right)\sqrt{\abs{b_n}},
    \]
    and $\setofunions_{n+1}=\set{P\in\setofunions:\
    b_{n+1}\notin P}$. We see that $\setofunions_{n+1}\subsetneq
    \setofunions_n$. Thus the construction can be
    continued inductively.

    Choose $P_n\in\setofunions_{n}\setminus
    \setofunions_{n+1}$ for each $n$; then $b_m\in P_n$ ($m>n$) but $b_m\notin P_n$ ($m\le n$). The compactness implies that
    there exists a pseudo-finite subsequence $(P_{n_i})$, whose
    union is denoted by $P$. For all $j>1$, $b_{n_j}\in P_{n_1}$
    so, $b_{n_j}\in P$. In particular, $b_{n_2}\in P$. On the other hand, for all $2\le i$, $b_{n_2}\notin
    P_{n_i}$, and so $b_{n_2}\notin P=\bigcup_{i=2}^\infty
    P_{n_i}$; a contradiction.
\end{proof}

\begin{proposition}
    \label{directdecomposition}
    Suppose that $\setofunions$ is a compact family of prime ideals in
    $\C_0(\Omega)$ with only one maximal element $Q$. Set
    $I=\bigcap\setofunions$. Let $P\in\setofunions$, and let $A$ be a subalgebra of
    $\C_0(\Omega)$. Suppose that $\C_0(\Omega)=A+P$ and $A\cap P$
    is the intersection of a sub-family of $\setofunions$. Let $B$ be a subalgebra of
    $A$ such that $B$ is maximal with respect to the property
    that $B\cap Q\subset I$. Then $\C_0(\Omega)=B+Q$.
\end{proposition}
\begin{proof}
    Note that $Q/I\subset \rad\C_0(\Omega)/I$; this follows from the previous lemma. Also
    that $I\subset B$, so indeed $B\cap Q=I$.

    \emph{Claim: for each $a\in A\setminus Q$
    and each $b\in A\cap Q$, there exists $s\in A\cap Q$ such that
    $as-b\in I$.} Indeed, by the previous lemma, there exists
    $s\in Q$ such that $as-b\in I$. Write $s=c+p$ where $c\in A$
    and $p\in P$. Then  $ap+(ac-b)\in I\subset A$ implies that $ap\in A\cap
    P$. Since $a\notin Q$ and $A\cap P$ is the intersection of
    a family of prime ideals contained in $Q$, we must have $p\in
    A\cap P$. Thus $s=c+p\in A\cap Q$.

    The above claim shows in particular that $(A\cap Q)/I \subset
    \rad A/I$. We shall prove that $A/I=B/I\oplus (A\cap Q)/I$; the proposition
    then follows.

    Assume toward a contradiction. Let $a\in A/I$ but $a\notin B/I\oplus (A\cap Q)/I$.
    By the maximality of
    $B$, there exists a non-zero polynomial $q(X)$ with coefficients in $B/I$
    such that $q(a)\in Q/I$. Let $q(X)$ be
    such a polynomial with smallest degree. Then $q'(a)\notin
    Q/I$, where $q'(X)$ is the formal derivative. Let $s\in \C_0(\Omega)/I$. Then
    \[
        q(a+q'(a)s)=q(a)+q'(a)^2s+\ldots+q'(a)^nq^{(n)}(a)\frac{s^n}{n!};
    \]
    where $q^{(k)}(X)$ is the formal $k^{\textrm{th}}$ derivative of
    $q(X)$.
    By the claim, there exists $d\in (A\cap Q)/I$ such that
    $q(a)=q'(a)^2d$. So
    \[
        q(a+q'(a)s)=q'(a)^2\left(d+s+\ldots+q'(a)^{n-2}q^{(n)}(a)\frac{s^n}{n!}\right).
    \]

    Since $\unitization{(A/(A\cap P))}\cong\unitization{(\C_0(\Omega)/P)}$ is Henselian, there exist
    $s\in \rad A/I$ such that
    \[
        d+s+\ldots+q'(a)^{n-2}q^{(n)}(a)\frac{s^n}{n!}\in (A\cap P)/I;
    \]
    (\cite{dales2000}, Theorem 2.4.30 and Proposition 1.6.3). It follows that $s\in
    (A\cap Q)/I$ and
    \[
        q(a+q'(a)s)\in (A\cap P)/I.
    \]
    If we set $b=a+q'(a)s$, then $b\in A/I$ but $b\notin B/I\oplus (A\cap Q)/I$,
    $q(b)\in (A\cap P)/I$, and $q(X)$
    is the smallest degree non-zero polynomial with coefficients in $B/I$
    such that $q(b)\in Q/I$. So without loss of generality, we can assume that $q(a)\in (A\cap
    P)/I$.

    The case where $q(a)\in (A\cap P)/I$: Then $d\in (A\cap P)/I$.
    Since $\unitization{\C_0(\Omega)}/I$ is Henselian, there exist
    $t\in\rad\C_0(\Omega)/I$ such that
    \[
        0=d+t+\ldots+q'(a)^{n-2}q^{(n)}(a)\frac{t^n}{n!}.
    \]
    Since $(A\cap P)/I$ is an ideal in $\C_0(\Omega)/I$,
    it follows that $t\in (A\cap P)/I$.
    Set $c=a+q'(a)t$. Then $c\in A/I$, $q(c)=0$, and $q(X)$
    is a non-zero polynomial with coefficients in $B/I$ with the smallest degree
    such that $q(c)\in Q/I$. Let $p(X)$ be any polynomial with coefficients in $B/I$
    such that $p(c)\in Q/I$. We see that there exist non-zero element $u\in
    B/I$ and a polynomial $h(X)$ with coefficient in $B/I$ such
    that $up(X)=q(X)h(X)$. Then $up(c)=q(c)h(c)=0$, and since
    $u\notin Q/I$, we deduce that $p(c)=0$. The maximality of $B$
    then implies that $c\in B/I$, and so $a=c-q'(a)t\in B/I+Q/I$; a contradiction.
\end{proof}

A special case of the previous proposition is when
$A=\C_0(\Omega)$.

\section{The main results}

\label{main_results}

In this section, we shall show the connection between continuity
ideals (as well as kernels) of homomorphisms from $\C_0(\Omega)$
into Banach algebras and intersections of (relatively) compact
families of prime ideals. One direction is an immediate
consequence of the results in section
\ref{abstract_continuity_ideals}, so most of this section concerns
the converse.

We shall need some basic complex algebraic-geometry results. Our
references for algebraic geometry will be \cite{kendig1977}. For a
set $S\subset\complexs[Z_1,Z_2,\ldots, Z_n]$, denote by
$\variety(S)$ the \emph{variety} (i.e., common zero set) of $S$ in
$\complexs^n$. For each prime ideal $Q$ in
$\complexs[Z_1,\ldots,Z_n]$, the variety $\variety(Q)$ is
\emph{irreducible}. The topology considered on complex spaces will
be the Euclidean topology. We shall need the fact that, for each
irreducible variety $V$ and each variety $W$ not containing $V$,
$V\setminus W$ is dense and (relatively) open in $V$ \cite[Chapter
IV, Theorem 2.11]{kendig1977}.

\begin{notation} For clarity, we shall use $X_i, Y_j$ for variables, $x_i, y_j$
for complex numbers, and $a_i, b_j$ for elements of an algebra.
When there is no ambiguity, we shall use boldface characters to
denote tuples of elements of the same type; for example, we set
\[
    \tuple{X} =(X_1, X_2,\ldots, X_m)\quad \textrm{or}\quad \tuple{y}=(y_1,\ldots,y_n)\,.
\]
In the case where $\tuple{X}=(X_1,\ldots,X_m)$, we also denote by
$\complexs_{\tuple{X}}$ the corresponding space $\complexs^m$.
\end{notation}

\begin{lemma}
    \label{fip_algebraicgeometry2}
    Let $m,n\in\naturals$, and let $Q$ be a prime ideal
    in $\complexs[\tuple{X},\tuple{Y}]$, where $\tuple{X}=(X_1,\ldots, X_m)$
    and $\tuple{Y}=(Y_1, \ldots, Y_n)$. Consider $Q_{\tuple{X}}=Q\cap\complexs[\tuple{X}]$
    as a prime ideal in $\complexs[\tuple{X}]$.
    Let $V$ be the variety of $Q$, and let $V_{\tuple{X}}$ be the variety of $Q_{\tuple{X}}$.
    Let $\pi$ be the natural projection
    $\complexs_{\tuple{X},\tuple{Y}}\to \complexs_{\tuple{X}}$.
    Then $\pi :V\to V_{\tuple{X}}$ and
    there exists a dense open subset $U$ of $V$ such that
     $\pi:U\to V_{\tuple{X}}$ is an open map.
\end{lemma}

\begin{proof}
Obviously, $\pi :V\to V_{\tuple{X}}$. Without loss of generality,
let $(X_1,\ldots, X_k)$ be a transcendental basis for
$\complexs[\tuple{X}]$ modulo $Q_{\tuple{X}}$. We consider
$\complexs^k=\complexs_{X_1,\ldots,X_k}$. Denote by $\pi_1$ the
natural projection $\complexs_{\tuple{X},\tuple{Y}}\to
\complexs^k$, and by $\pi_2$ the natural projection
$\complexs_{\tuple{X}}\to \complexs^k$.

By \cite[Lemma 6.3]{pham2005b}, there exist dense open subsets $U$
and $U_{\tuple{X}}$ of $V$ and $V_{\tuple{X}}$, respectively, such
that $\pi_1:U\to \complexs^k$ and
$\pi_2:U_{\tuple{X}}\to\complexs^k$ are open maps. Inspecting the
proof of \cite[Lemma 6.3]{pham2005b}, we see that $U_{\tuple{X}}$
can be chosen as $V_{\tuple{X}}\setminus V_0$, where $V_0$ is a
proper subvariety of $V_{\tuple{X}}$, and that $\pi_2$ is even a
local homeomorphism from $U_{\tuple{X}}$ onto an open subset of
$\complexs^k$. Since $V_0$ has dimension at most $k-1$
\cite[Chapter IV]{kendig1977}, we can further require that
\[
    \pi_2(U_{\tuple{X}})\cap\pi_2(V_{\tuple{X}}\setminus U_{\tuple{X}})=\emptyset.
\]

Let $W=\pi_1(U)\cap\pi_2(U_{\tuple{X}})$. Then $W$ is an open set
in $\complexs^k$. It can be seen that $W$ is dense in $\pi_1(U)$.
Set
    \[
        U'=U\cap\pi_1^{-1}(W).
    \]
Then $U'$ is a dense open subset of $V$. Let
$(\tuple{x},\tuple{y})\in U'$. We can see that $\pi:U'\to
V_{\tuple{X}}$ is an open map.
\end{proof}

\begin{proposition}
    \label{fip_mainalgebraictheorem}
    Let $A=\C_0(\Omega)$ for a locally compact space $\Omega$,
    and let $\setofunions$ be a non-empty compact family of
    non-modular prime ideals in $A$. Suppose that each chain in
    $\setofunions$ is countable.
    Then there exist a cardinal $\kappa$, a free ultrafilter $\U$ on
    $\kappa$, and, for each $P\in \setofprimes$, a homomorphism
    $\theta_P :A\to\infinitesimals{\ultrapower}$ such that:
    \begin{itemize}
        \item[(a)] $\ker \theta_P=P$ $(P\in \setofprimes)$, and
        \item[(b)] the set $\set{\theta_P(a): P\in\setofprimes}$ is finite for each $a\in A$.
    \end{itemize}
\end{proposition}

\begin{proof}
Note that $A$ must be non-unital. Since each $P\in\setofunions$ is
a non-modular prime ideal in $A$, it is a prime ideal in $\unitA$.
For each $Q\in\setofunions$, set
\[
    \setofunions_Q=\set{P\in\setofunions:\ P\subset Q},
\]
and set $I_Q=\bigcap \setofunions_Q$. We start the proof with some
lemmas:

\newcommand{\achain}{\mathfrak{C}}

\begin{lemma}
  Let $Q_*\subset Q^*\in\setofunions$. Let $A_*\supset A^*$ be subalgebras of $A$ such that
  $A^*\cap Q^*\subset I_{Q^*}$, and $A_*\cap Q_*= I_{Q_*}$.
  Suppose further that $A_*+Q_*=A$.
  Let $\achain$ be a chain in $\setofunions_{Q^*}$ where each
  ideal in $\achain$ contains $Q_*$.
  Then, for each $Q\in\achain$, we can find a subalgebra $A_Q\subset
  A$, such that the following conditions is satisfied:
  \begin{enumerate}
    \item $A_Q\cap Q=I_Q$ and $A=A_Q+Q$ $(Q\in\achain)$;
    \item for each $Q_1\subset Q_2\in\achain$, we have $A_*\supset A_{Q_1}\supset
    A_{Q_2}\supset A^*$.
  \end{enumerate}
\end{lemma}
\begin{proof}
  Assume toward a contradiction that the lemma fails.
  Let $Q_*, Q^*$, $A_*, A^*$, and $\achain$ be as above and such that the lemma fails and that
  $\order(\achain)$ is smallest. Proposition
  \ref{directdecomposition} implies that $\achain$ must be
  infinite.

  If $\achain$ is order isomorphic to $\omega$ the first infinite
  ordinal; say
  \[
  \achain=\set{Q_1\subset Q_2\subset\cdots}.
  \]
  Then, Proposition \ref{directdecomposition}
  enable us to construct $(A_{Q_n})$ inductively.

  In general, we index $\achain$ increasingly by $\alpha\in\order(\achain)$, so that
  $\achain=(Q_\alpha)$.
  There exists a sequence $(\alpha_n)$ in $\order(\achain)$ converging
  in the order topology to $\gamma=\sup\order(\achain)$.
  If $\gamma\in\order(\achain)$, we can, by Proposition \ref{directdecomposition}
  construct $A_{Q_\gamma}$ first.
  As in the previous paragraph, we can (then) find $A_{Q_{\alpha_n}}$ ($n\in\naturals$)
  satisfying both the conditions (i) and (ii). The $\alpha_n$'s
  divide $\order(\achain)\setminus\set{\gamma}$ into chains which are
  order isomorphic to ordinals strictly
  smaller than $\order(\achain)$. The minimality of
  $\order(\achain)$ then allows us to find $A_Q$ ($Q\in\achain$) satisfying the
  conditions (i) and (ii).

  Thus, in any case, we have a contradiction.
\end{proof}

\begin{lemma}
  For each $Q\in\setofunions$, we can find a subalgebra $A_Q\subset
  A$ such that the following conditions is satisfied:
  \begin{enumerate}
    \item $A_Q\cap Q=I_Q$ and $A=A_Q+Q$ $(Q\in\setofunions)$;
    \item for each $Q_1\subset Q_2\in\setofunions$, we have $A_{Q_1}\supset
    A_{Q_2}$.
  \end{enumerate}
\end{lemma}
\begin{proof}
    This follows from Zorn's lemma, the fact that all prime ideals
    containing a given prime ideal form a chain, and the previous lemma.
\end{proof}

Let $\kappa$ be the set of all tuples of the form
$(\delta;\fsetofunions; a_1,\ldots,a_m)$, where $\delta
>0$, $\fsetofunions$ is a non-empty finite subsets of $\setofunions$,
and $(a_1,\ldots,a_m)$ is a non-empty finite sequence of distinct
elements in $A$. Define a partial order $\prec$ on $\kappa$ by
setting
\[
    (\delta;\fsetofunions; a_1,a_2,\ldots,a_m)
    \prec(\delta';\fsetofunions'; a'_1,a'_2,\ldots,a'_{m'})
\]
if $\delta >\delta'$, $\fsetofunions\subset \fsetofunions'$,
$(a_1,\ldots,a_m)$ is a subsequence of
$(a'_1,a'_2,\ldots,a'_{m'})$. Then $(\kappa,\prec)$ is a net. Fix
an ultrafilter $\U$ on $\kappa$ majorizing this net.

\begin{lemma}
    \label{algebraicconstruction}
    For each $w=(\delta;\fsetofunions; a_1,\ldots,a_m)\in \kappa$.
    Then we can find, for each $P\in\fsetofunions$, a finite sequence of complex numbers
    \[\tau_P(w)=\tuple{x}^{(P)}=(\xiota_1,\ldots, \xiota_m)\]
    satisfying all the following conditions:
    \begin{enumerate}
        \item $p(\xiota_1,\ldots,\xiota_m)=0$ for each $p\in\complexs[X_1,\ldots,X_m]$
            with $p(a_1,\ldots,a_m)\in P$;
        \item for each $1\le k\le m$ with $a_k\notin P$,
            we have $\xiota_k\neq 0$;
        \item $|\xiota_k|\le \delta\quad(1\le k\le m)$.
        \item for each $1\le k\le m$ and each $P\subset Q\in\fsetofunions$
        such that $a_k\in A_Q$, we have $\xiota_k=x^{(Q)}_k$.
    \end{enumerate}
\end{lemma}

\begin{proof}
For each $P\in\fsetofunions$, set
$\tuple{a}^P=(a_1,\ldots,a_m)\cap A_P$ and set $\tuple{X}^P=(X_i:\
a_i\in A_P)$. Without loss of generality, we can assume that none
of these is empty. Note that, when $P\subset Q\in\fsetofunions$
then $\tuple{a}^Q\subset\tuple{a}^P$ and
$\tuple{X}^Q\subset\tuple{X}^P$.

Denote by $\pi_P$ the projection from $\complexs_{\tuple{X}}$ onto
$\complexs_{\tuple{X}^P}$, and $\pi_{PQ}$ the projection from
$\complexs_{\tuple{X}^P}$ onto $\complexs_{\tuple{X}^Q}$
($P\subset Q\in\fsetofunions$). We also conveniently consider
$\complexs_{\tuple{X}^P}$'s as subspaces of
$\complexs_{\tuple{X}}$.

For each $Q\in\fsetofunions$, define
$\widetilde{Q}=\set{p\in\complexs[\tuple{X}] :\, p(\tuple{a})\in
Q}$, and
\[
    \widehat{Q}=\set{p\in\complexs[\tuple{X}^Q] :\  p(\tuple{a}^Q)\in Q}\,.
\]
We see that for $P\subset Q\in \fsetofunions$,
$\widetilde{P}\subset \widetilde{Q}$ are prime ideals in
$\complexs[\tuple{X}]$, and
$\widehat{Q}=\widetilde{Q}\cap\complexs[\tuple{X}^Q]$ is a prime
ideal in $\complexs[\tuple{X}^Q]$. Also, since $A_Q\cap Q=A_Q\cap
P$ for each $P\subset Q\in\fsetofunions$, we see that
\[
    \widehat{Q}=\set{p\in\complexs[\tuple{X}^Q] :\  p(\tuple{a}^Q)\in
    P}=\widetilde{P}\cap\complexs[\tuple{X}^Q];
\]
and thus $\widehat{Q}=\complexs[\tuple{X}^Q]\cap\widehat{P}$.

It follows from Lemma \ref{fip_algebraicgeometry2} that there
exist, for each $Q\in\fsetofunions$, dense open subsets $U_Q$ of
$\variety(\widetilde{Q})$ and $W_Q$ of $\variety(\widehat{Q})$,
respectively, such that:
\begin{itemize}
    \item $\pi_Q:U_Q\to W_Q$ is an open map;
    \item $\pi_{PQ}:W_P\to W_Q$ is an open map ($P\subset
    Q\in\fsetofunions$).
\end{itemize}
(We have only finitely many varieties here.)

For each $Q\in\fsetofunions$, set
\[
    V_{Q}= \bigcup_{1\le r\le m,\ a_r\notin Q}\set{\tuple{x}=(x_1,\ldots,x_m):\ x_r = 0}\,.
\]
Then $V_{Q}$ is a variety which does not contain
$\variety(\widetilde{Q})$, and so
$\variety(\widetilde{Q})\setminus V_{Q}$ is also a dense open
subset of $\variety(\widetilde{Q})$ because
$\variety(\widetilde{Q})$ is irreducible. Therefore,
$U_{Q}\setminus V_{Q}$ is again a dense open subset of
$\variety(\widetilde{Q})$.

Set
\[
    \Delta=\set{\tuple{x}=(x_1,\ldots,x_m):\ |x_r|< \delta\ (1\le r\le m)}\,.
\]
Finally, set $U'_{Q}=(U_{Q}\setminus V_{Q})\cap \Delta$ and
$W'_Q=W_Q\cap \Delta\,$. Note that the origin $\tuple{0}$ is in
$\variety(\widetilde{Q})$ and $\variety(\widehat{Q})$
($Q\in\fsetofunions$). So $U'_Q$ (respectively, $W'_{Q}$) is a
non-empty open subset of $\variety(\widetilde{Q})$ (respectively,
$\variety(\widehat{Q})$).

Choosing any $\tuple{x}^{(Q)}\in U'_Q$ will ensure the conditions
(i), (ii), and (iii) are satisfied.

Fix $P\subset Q\in\fsetofunions$. Since $U_P\setminus V_P$ is
dense in $\variety(\widetilde{P})$ and $U'_Q\subset
\variety(\widetilde{Q})\subset \variety(\widetilde{P})$, we have
$\pi_{Q}(U'_P)\cap \pi_Q(U'_Q)$ is dense (and open) in
$\pi_{Q}(U'_Q)$ (which is open in $W_Q$). Note that
$\pi_{Q}(U'_P)=\pi_{PQ}[\pi_{P}(U'_P)]$. In fact, we see that for
every dense open subset $D$ of $\pi_P(U'_P)$,
$\pi_{PQ}(D)\cap\pi_Q(U'_Q)$ is dense and open in $\pi_Q(U'_Q)$.

Thus, we can define dense open subsets $D_Q$ of $\pi_Q(U'_Q)$ as
follows: For $P$ minimal in $\fsetofunions$, set
$D_P=\pi_P(U'_P)$. Then, define inductively,
\[
    D_Q=\bigcap_{P\in\fsetofunions,\ P\subsetneq Q}\pi_{PQ}(D_P)\cap\pi_Q(U'_Q).
\]

The non-emptiness of $D_P's$ allows us to choose
$\tuple{x}^{(Q)}\in U'_Q$ such that the condition (iv) is
satisfied (inductively, starting from the maximal elements in
$\fsetofunions$; note that the prime ideals containing a given
prime ideal form a chain).
\end{proof}

For each $P\in\setofunions$, define $\xi_P:A\to \complexs^\kappa$
as follows: For each $a\in A$ and each
\[
    w =(\delta;\fsetofunions;a_1,\ldots,a_m)\in\kappa,
\]
if $P\in\fsetofunions$ and if $a$ is in $(a_1,\ldots,a_m)$, say $a=a_k$ (there is at most one
such $k$), then $\xi_P(a)(w)= \tau_P(w)(k)$; otherwise,
$\xi_P(a)(w)=0$. Define $\theta_P(a)$ to be the equivalence class
in $\ultrapower$ containing $\xi_P(a)$.

Let $P\in \setofunions$. By Lemma \ref{algebraicconstruction}(i),
the map $\theta_P$ is an algebra homomorphism and $P\subset\ker
\theta_P$; by combining this with Lemma
\ref{algebraicconstruction}(ii), we see that $\ker \theta_P$ is
exactly $P$. By Lemma \ref{algebraicconstruction}(iii), the image
of $\theta_P$ is contained in $\infinitesimals{\ultrapower}$. By
Lemma \ref{algebraicconstruction}(iv), for each $P\subset
Q\in\setofunions$ and each $a\in A_Q$, we have
$\theta_P(a)=\theta_Q(a)$.

Let $(P_n)$ be any pseudo-finite sequence in $\setofunions$, and
let $Q=\bigcup_{n=1}^\infty P_n$. Let $a\in A$. Write $a=b+x$,
where $b\in A_Q$ and $x\in Q$. Then $x\in P_n$ for $n>n_0$, for
some $n_0$. Therefore, for $n>n_0$,
$\theta_{P_n}(a)=\theta_{P_n}(b)=\theta_Q(b)$. Thus, the sequence
$(\theta_{P_n}(a))$ is eventually constant.

Finally, assume toward a contradiction that the set
$\set{\theta_P(a): P\in\setofunions}$ is infinite for some $a\in
A$. Let $(P_n)$ be a sequence in $\setofunions$ such that all
$\theta_{P_n}(a)$ ($n\in\naturals$) are distinct. By the
compactness of $\setofunions$, $(P_n)$ contains a pseudo-finite
subsequence $(P_{n_i})$. However, the previous paragraph shows
that the set $\set{\theta_{P_{n_i}}(a):i\in\naturals}$ is finite.
This is a contradiction.

This finished the proof of Proposition
\ref{fip_mainalgebraictheorem}.
\end{proof}

\begin{remark}
The idea of the above proof originate from the approach in
\cite{daleswoodin1996} of the theorem \cite{esterle1979} of
Esterle on embedding integral domains into radical Banach
algebras.
\end{remark}

We are now ready to prove our main results.

\begin{theorem}
    \label{main_theorem_radical}
    Let $\Omega$ be a locally compact space.
    \begin{enumerate}
        \item Let $\theta$ be a homomorphism from $\C_0(\Omega)$ into a radical Banach algebra
            $R$. Then $\ker\theta$ is the intersection of a relatively compact
            family of non-modular prime ideals in $\C_0(\Omega)$.
        \item \usingCH Let $I$ be the intersection of a relatively compact family of
            non-modular prime ideals in $\C_0(\Omega)$ such that $I$ is also the intersection
        of a countable family of prime ideals. Suppose that
            \[
                \cardinal{\C_0(\Omega)\big/I}=\continuum\,.
            \]
            Then there exists a homomorphism $\theta$ from $\C_0(\Omega)$ into a radical
            Banach algebra such that $\ker\theta=I$.
    \end{enumerate}
\end{theorem}
\begin{proof}
    (i) Since $\theta$ maps into a radical algebra, we see that $\quotoel{\ker\theta}{f}$
    is non-modular
    for each $f\in\C_0(\Omega)$. In this case, Theorem \ref{fip_bcse} shows that
    $\ker\theta=\cont(\theta)$, and so, it is an abstract continuity in $\C_0(\Omega)$.
    The result will then follow from Corollary \ref{main_theorem_abstract}(i).

    (ii) Let $\setofprimes$ be a relatively compact family of non-modular
    prime ideals in $\C_0(\Omega)$ such that $I=\bigcap\setofprimes$.
    We can assume that $\setofprimes\neq\emptyset$. By keeping only
    the minimal elements in $\setofprimes$, we can suppose that
    $P\nsubset Q$ ($P\neq Q\in\setofprimes$). Since $I$ is the
    intersection of a countable family of prime ideals, by Lemma
    \ref{nonredundancy}, we see that $\setofprimes$ is countable. Let
    $\setofunions$ be the set of all the ideals that are unions of
    (countably many) prime ideals in $\setofprimes$. It is easy to see
    that every chain in $\setofunions$ is countable.

Let $\theta_P:\C_0(\Omega)\to\infinitesimals{\ultrapower}$
($P\in\setofunions$) be homomorphisms as in Proposition
\ref{fip_mainalgebraictheorem}. Let $B$ be the subalgebra of
$\infinitesimals{\ultrapower}$ generated by all the images of
$\theta_P$ ($P\in \setofunions$). Then $B$ is a non-unital
integral domain. We also have
\[
    \cardinal{B}=\cardinal{\bigcup_{P\in \setofunions} \theta_P(\C_0(\Omega))}
    =\cardinal{\bigcup_{a\in \C_0(\Omega)}\set{\theta_P(a):\ P\in \setofunions}}
    =\continuum\,;
\]
since, for each $b\in a+I$, we have
\[
    \set{\theta_P(b):\ P\in \setofunions}=\set{\theta_P(a):\ P\in
    \setofunions},
\]
which is finite. Thus \cite{esterle1979} there exists an embedding
$\psi : B\to R_0$ where $R_0$ is a universal radical Banach
algebra; for example, $R_0=\Lone(\reals^+,\omega)$ for $\omega$ a
radical weight bounded near the origin.

Consider the following map:
\[
    \theta:\C_0(\Omega)\to\prod_{P\in \setofunions} R_0\,,\ a\mapsto
    \big((\psi\circ\theta_P)(a):\ P\in \setofunions\big)\,.
\]
Then $\theta$ is a homomorphism with kernel $\bigcap
\setofunions=I$. We see, by Proposition
\ref{fip_mainalgebraictheorem}, that the image of $\theta$ is in
$\linfty(\setofunions,R_0)$, and indeed is in its radical $R$.
Thus $\theta :\C_0(\Omega)\to R$ is the required homomorphism.
\end{proof}

\begin{theorem}
    \label{main_theorem}
    Let $\Omega$ be a locally compact space.
    \begin{enumerate}
        \item Let $\theta$ be a homomorphism from $\C_0(\Omega)$ into a Banach algebra
            $B$. Then $\cont(\theta)$ is the intersection of a relatively compact
            family of prime ideals in $\C_0(\Omega)$.
        \item \usingCH Let $I$ be the intersection of a relatively compact family of
            prime ideals in $\C_0(\Omega)$ such that $I$ is also the intersection of a countable
        family of prime ideals. Suppose that
            \[
                \cardinal{\C_0(\Omega)\big/I}=\continuum\,.
            \]
            Then there exists a homomorphism $\theta$ from $\C_0(\Omega)$ into a
            Banach algebra such that $\cont(\theta)=I$.
    \end{enumerate}
\end{theorem}
\begin{proof}
    (i) The continuity ideal $\cont(\theta)$ is an abstract continuity ideal in $\C_0(\Omega)$.
    The result will follow from Corollary \ref{main_theorem_abstract}(i).

    (ii) Let $\setofprimes$ be a relatively compact family of prime ideals
in $\C_0(\Omega)$ such that $I=\bigcap\setofprimes$. We can assume
that $\setofprimes\neq\emptyset$. As in the previous proof, we can
assume that $\setofprimes$ is countable.

Denote by $\setofprimes_0$ the set of non-modular ideals in
$\setofprimes$. Let $\setofunions'$ be the collection of all the
ideals that are unions of (countably many) prime ideals in
$\setofprimes\setminus\setofprimes_0$. By Lemma \ref{finiteroofs},
$\setofunions'$ has only finitely many roofs. Denote by
$Q_1,\ldots, Q_n$ the roofs of $\setofunions'$, and set
$\setofprimes_i=\set{P\in\setofprimes\setminus\setofprimes_0:\
P\subset Q_i}$.

Let $1\le k\le n$. First, it is easy to see that $Q_k$ is a
modular prime ideal. Pick a modular identity $u$ for $Q_k$, and
pick $a\notin Q_k$. Then $a-au\in Q_k$, and so, by Lemma
\ref{divisibility}, there exists $v\in Q_k$ such that
$a-au-av\in\bigcap\setofprimes_k$. It follows easily that $u+v$ is
a modular identity for $\bigcap\setofprimes_k$; denote it by
$u_k$.

Theorem \ref{main_theorem_radical} shows that there exists a
homomorphism $\theta_0$ from $\C_0(\Omega)$ into a radical Banach
algebra $R_0$ such that $\ker\theta_0=\bigcap\setofprimes_0$.
Similarly, for each $1\le k\le n$, there exists a homomorphism
$\theta_k$ from $M_k$ into $R_k$ such that
$\ker\theta_k=\bigcap\setofprimes_k$; where $M_k$ is the maximal
modular ideal containing $Q_k$. We extend $\theta_k$ to a
homomorphism from $\C_0(\Omega)$ into $\unitization{R}$ by setting
$\theta_k(u_k)=\unit_{R_k}$. It still remains that
$\ker\theta_k=\bigcap\setofprimes_k$.

It follows from the result of Bade and Curtis that
$\cont(\theta_k)=\ker\theta_k$ ($0\le k\le n$). Thus the
homomorphism $\theta:\C_0(\Omega)\to\prod_{k=0}^n
\unitization{R_k}$ defined by
$\theta(a)=(\theta_0(a),\ldots,\theta_n(a))$ satisfies
$\cont(\theta)=\bigcap_{k=0}^n\cont(\theta_k)=\bigcap\setofprimes=I$.
\end{proof}

\begin{remark}
    In Parts (ii) of Theorems \ref{main_theorem_radical} and
    \ref{main_theorem}, we only need that $I$ is the intersection of a relatively compact
    family $\setofprimes$ of prime ideals where every chain in the
    closure of $\setofprimes$ is countable (see Proposition \ref{fip_mainalgebraictheorem}).
\end{remark}

\section{Examples on metrizable locally compact spaces}

For examples of pseudo-finite sequence of prime ideals in
$\C_0(\Omega)$ for $\Omega$ metrizable, see \cite{pham2005b}.
Obviously, unions of finitely many pseudo-finite families are
relatively compact. In this section, we shall construct relatively
compact families of prime ideals that are not unions of finitely
many pseudo-finite families.

Let $\kappa$ be a well-ordered set. Set
$\orderlevel{\kappa}{0}=\kappa$. For each $n\in\naturals$, define
inductively $\orderlevel{\kappa}{n}$ as the set of limiting
elements in $\orderlevel{\kappa}{n-1}$. \emph{We shall only
consider those $\kappa$ for which
$\orderlevel{\kappa}{n}=\emptyset$ for some $n\in\naturals$}. This
condition force $\kappa$ to be countable. Lets call the largest
integer $d$ for which $\orderlevel{\kappa}{d}\neq\emptyset$ the
\emph{depth} of $\kappa$. \emph{For simplicity, we also suppose
that $\kappa$ has the largest element, $\max\kappa$, and that
$\orderlevel{\kappa}{d}=\set{\max\kappa}$.} Otherwise, we can
always replace $\kappa$ by a bigger well-ordered set.

For each $\alpha\in\kappa$ define $l(\alpha)$ to be the largest
integer $l$ for which $\alpha\in\orderlevel{\kappa}{l}$. We define
a partial order $\prec$ on $\kappa$ as follows: For each
$\alpha,\beta\in\kappa$, we write $\alpha\prec\beta$ if $\beta$ is
the smallest element in $\kappa$ with the properties that
$\beta\ge \alpha$ and that $l(\beta)=l(\alpha)+1$. We define
another partial order $\ll$ on $\kappa$ as follows: For each
$\alpha,\beta\in\kappa$, we write $\alpha\ll\beta$ if there exists
a finite sequence
$\alpha=\gamma_1\prec\gamma_2\prec\ldots\prec\gamma_n=\beta$.
\emph{Note that} if $\beta\ll\alpha$ and $\beta\le\gamma\le\alpha$
then $\gamma\ll\alpha$.

\begin{lemma}
    \label{idealsconstruction}
    Let $A$ be a commutative algebra and $Q$ be an ideal which either is prime
    in $A$ or is $A$ itself. Suppose that we have
    $(f_\alpha:\alpha\in\kappa)\subset Q$ and a semiprime ideal
    $I\subset Q$ such that
    \begin{enumerate}
        \item $f_\alpha\notin I$ and $\quotoel{I}{f_\alpha}\subset Q$;
        \item $f_\alpha f_\beta\in I$ whenever both $\alpha\not\ll\beta$
        and $\beta\not\ll\alpha$;
        \item if $gf_\alpha\in I$ then $gf_\beta\in I$ for all
        $\beta\ll\alpha$;
        \item if $l(\beta_0)=0$, $\beta_0\ll\alpha$, $\beta_0\neq\alpha$ and $gf_{\beta_0}\in I$ then
        there exists $\beta_1\ll\alpha$ such that $l(\beta_1)=0$ and
        that $gf_\beta\in I$ for all $\beta_1\le\beta\le\alpha$ with $l(\beta)=0$.
    \end{enumerate}
    Then there exist prime ideals $(P_\alpha:\alpha\in\kappa)$
    satisfying that:
    \begin{itemize}
        \item[(a)] $f_\alpha\notin P_\alpha$ and $\quotoel{I}{f_\alpha}\subset P_\alpha\subset Q$;
        \item[(b)] $P_\alpha=\bigcup_{\beta\ll\alpha,\beta\neq\alpha} P_\beta$;
        \item[(c)] if $g\in P_{\alpha}$ then there exists $\beta_1\ll\alpha$
      such that $l(\beta_1)=0$ and that $g\in P_\beta$ for all
      $\beta_1\le\beta\le\alpha$ with $l(\beta)=0$.
    \end{itemize}
\end{lemma}
\begin{proof} We prove by induction on the depth $d$ of $\kappa$.

When $d=0$, $\kappa=\set{0}$. The conditions (i)-(iv) reduce to
$I$ being semiprime, $f_0\notin I$ and $\quotoel{I}{f_0}\subset Q$.
It follows that $I\cap\set{f_0^k, f_0^kf:\ k\ge 1, f\in A\setminus
Q}=\emptyset$. Therefore, we can find a prime ideal $P_0$ such
that $I_0\subset P_0\subset Q$ and $f_0\notin P_0$. We see that
$P_0$ is the required prime ideal.

Now, suppose that the result hold for all the depth less than $d$.
By Zorn's lemma, we can choose a semiprime ideal $J$ containing
$I$ such that $J$ is maximal with respect to conditions (i)-(iv).

\emph{Claim 1: If $f\notin Q$ then $\quotoel{J}{f}=J$.} Indeed, it
is easy to see that $\quotoel{J}{f}$ is semiprime and satisfies
conditions (i)-(iv). So the maximality of $J$ implies
$\quotoel{J}{f}=J$.

\emph{Claim 2: If $f\notin J$ then $\quotoel{J}{f}\subset Q$.} For
otherwise, there would exist $g\in\quotoel{J}{f}\setminus Q$, and
so $f\in\quotoel{J}{g}=J$, by Claim 1; a contradiction.

Set $P=\bigcup_{\alpha\in\kappa} \quotoel{J}{f_\alpha}$. Then
$P\subset Q$. Condition (iii) implies that
\[
    P=\bigcup_{\alpha\in\kappa,\ l(\alpha)=0} \quotoel{J}{f_\alpha}
\]
and condition (iv) implies that $P$ is an ideal.

\emph{Claim 3: If $f\notin P$ then $\quotoel{J}{f}=J$.} Indeed, it
is easy to see that $\quotoel{J}{f}$ is semiprime and satisfies
conditions (i)-(iv) (the less obvious one is (i), however, since
$f\notin P$, $ff_\alpha\notin J$, and so $f_\alpha\notin
\quotoel{J}{f}$ and
$\quotoel{(\quotoel{J}{f})}{f_\alpha}=\quotoel{J}{ff_\alpha}\subset
Q$ by Claim 2). So the maximality of $J$ implies
$\quotoel{J}{f}=J$.

\emph{Claim 4: $P$ is either prime in $A$ or $A$ itself.} Indeed,
if $f,g\notin P$, then, by Claim 3,
\[
    g\notin \bigcup_{\alpha\in\kappa}
    \quotoel{J}{f_\alpha}=\bigcup_{\alpha\in\kappa}
    \quotoel{(\quotoel{J}{f})}{f_\alpha}=\quotoel{P}{f}.
\]
Thus $fg\notin P$.

Let $\alpha_1<\alpha_2<\ldots$ be the non-limiting elements in
$\orderlevel{\kappa}{d-1}$; their limit is $\max\kappa$. Set
$\kappa_1=\set{\alpha\in\kappa:\ \alpha\le\alpha_1}$, and, for
each $n\ge 2$, set $\kappa_n=\set{\alpha\in\kappa:\
\alpha_{n-1}<\alpha\le\alpha_{n}}$. Each $\kappa_n$ has depth
$d-1$, and
$\kappa=\bigcup_{n=1}^\infty\kappa_n\cup\set{\max\kappa}$.

For each $n\in\naturals$, we see that $(f_\alpha:\
\alpha\in\kappa_n)$, $J$, and $P$ satisfy $(f_\alpha:\
\alpha\in\kappa_n)\subset P$, $J\subset P$, and conditions
(i)-(iv) (with $\kappa_n$ replacing $\kappa$, $J$ replacing $I$,
and $P$ replacing $Q$). So, by induction, there exist prime ideals
$P_{\alpha}$ ($\alpha\in\kappa_n$) satisfying the conditions
(a)-(c) (with obvious modification). Set $P_{\max\kappa}=P$.

Note that if $\beta\ll\alpha<\max\kappa$ then both $\alpha$ and
$\beta$ belong to the same $\kappa_n$ for some $n\in\naturals$. We
see that the combined sequence $(P_\alpha:\ \alpha\in\kappa)$
obviously satisfies the conditions (a)-(c) (with $J$ replacing
$I$); the only one need to really check is condition (c) when
$\alpha=\max\kappa$, however, this case follows from the facts
that $\quotoel{J}{\beta}\subset P_\beta\subset P_{\max\kappa}$,
that
\[
    P_{\max\kappa}=\bigcup_{\beta\in\kappa,\ l(\beta)=0}\quotoel{J}{f_\beta},
\]
and that $J$ satisfies condition (iv). \end{proof}

Now, let $\Omega$ be a metrizable locally compact space. We define
a non-increasing sequence
$\left(\level{\onepointcompactification{\Omega}}{n}:n\in\integers^+\right)$
of compact subsets of $\onepointcompactification{\Omega}$ as
follows:
\begin{enumerate}
    \item put $\level{\onepointcompactification{\Omega}}{0}=\onepointcompactification{\Omega}$;
    \item for each $n\in\integers^+$, define
    $\level{\onepointcompactification{\Omega}}{n+1}$ to be the set of all limit points of
    $\level{\onepointcompactification{\Omega}}{n}$.
\end{enumerate}
Define $\level{\onepointcompactification{\Omega}}{\infty}
=\bigcap\set{\level{\onepointcompactification{\Omega}}{n}:\,
n\in\integers^+}$. By the compactness, either
$\level{\onepointcompactification{\Omega}}{\infty}$ is non-empty
or $\level{\onepointcompactification{\Omega}}{l}$ is empty for
some $l\in\integers^+$.

In what follows, the hypothesis that
$p\in\level{\onepointcompactification{\Omega}}{\infty}$ is
necessary, because of \cite[Proposition 8.7]{pham2005b}.

\begin{theorem}
    \label{examples}
    Let $\Omega$ be a metrizable locally compact space, and
    let $p\in\level{\onepointcompactification{\Omega}}{\infty}$. Let $\kappa$ be a
    well-ordered set as above. Then, there exist a sequence of prime
    ideals $(P_\alpha:\ \alpha\in\kappa)$ in
    $\C_0({\Omega})$, where each ideal is
    supported at $p$, and a sequence of functions
    $(f_\alpha:\ \alpha\in\kappa)$ in $\C_0({\Omega})$ satisfying the following:
    \begin{itemize}
        \item[(a)] $f_\alpha\notin P_\alpha$ and $f_\beta\in P_\alpha$ whenever
            both $\beta\not\ll\alpha$ and $\alpha\not\ll\beta$;
        \item[(b)] $P_\alpha=\bigcup_{\beta\ll\alpha,\beta\neq\alpha} P_\beta$;
        \item[(c)] if $g\in P_{\alpha}$ then there exists $\beta_1\ll\alpha$ such that
      $l(\beta_1)=0$ and that $g\in P_\beta$ for all $\beta_1\le\beta\le\alpha$
      with $l(\beta)=0$.
    \end{itemize}
\end{theorem}
\begin{proof}
Set $\kappa_0=\set{\beta\in\kappa:\ l(\beta)=0}$. For each
$\beta\in\kappa\setminus\set{\max\kappa}$, there exists a unique
$\alpha\in\kappa$ such that $\beta\prec\alpha$; the set
$\set{\gamma:\ \gamma\prec\alpha}$ is order isomorphic to
$\naturals$, and so we can define $t(\beta)$ to be the natural
number corresponding to $\beta$. For each $\beta\in\kappa_0$, there
exists a unique sequence $(\alpha_1,\ldots,\alpha_{d-1})\in\kappa$
such that
\[
\beta\prec\alpha_1\prec\cdots\prec\alpha_{d-1}\prec\max\kappa;
\]
set
$w(\beta)=\max\set{t(\beta),t(\alpha_1),\ldots,t(\alpha_{d-1})}$.
Note that, for each $k\in\naturals$,
\[
 \cardinal{\set{\alpha\in\kappa_0:\ w(\alpha)\le k}}=k^d.
\]

Adjoin $\infty$ to $\naturals$ to obtain its one-point
compactification $\onepointcompactification{\naturals}$. Define
$\Xi$ to be the subset of the product space
$(\onepointcompactification{\naturals})^{\kappa_0}$ consisting of
all elements $(n_\alpha:\ \alpha\in\kappa_0)$  with the properties
that there exists a finite set $F\subset\kappa_0$ such that
$n_\alpha=\infty$ ($\alpha\in\kappa_0\setminus F$) and such that
$n_\alpha\ge \max\set{w(\beta): \beta\in F}$
($\alpha\in\kappa_0$). It is easy to see that $\Xi$ is a closed
subset of $(\onepointcompactification{\naturals})^{\kappa_0}$.
Since $\kappa$ must be countable, the space $\Xi$ can be embedded
into $\onepointcompactification{\Omega}$ such that
$\tuple{\infty}=(\infty,\infty,\ldots)$ is mapped into $p$
(similar to the proofs of \cite[Lemmas 9.1 and 9.2]{pham2005b}).
Thus we only need to find a system of prime ideals $(P_\alpha:\
\alpha\in\kappa)$ and functions $(f_\alpha:\ \alpha\in\kappa)$
satisfying condition (a)-(c) in $\C(\Xi)$ such that all $P_\alpha$
are supported at $\tuple{\infty}$.

For each $\alpha\in\kappa_0$, define
\[
    Z_\alpha=\set{(j_\beta)_{\beta\in\kappa_0}\in\Xi:\
    j_\alpha=\infty},
\]
and for $\alpha\in\kappa\setminus\kappa_0$, define
\[
    Z_\alpha=\bigcap_{\beta\ll\alpha,\ l(\beta)=0}Z_\beta.
\]
Then choose $f_\alpha\in\C(\Xi)$ such that
$Z_\alpha=\zero(f_\alpha)$. Let $\F$ to be the $z$-filter
generated by all $Z_\alpha\cup Z_\beta$ ($\alpha,\beta\in\kappa,\,
\alpha\not\ll\beta$ and $\beta\not\ll\alpha$). Then define
$I=\zero^{-1}[\F]$. Obviously $I$ is a semiprime ideal, $I\subset
M_{\tuple{\infty}}$, and $(f_\alpha:\ \alpha\in\kappa)\subset
M_{\tuple{\infty}}$.

We shall prove that $(f_\alpha)$, $I$, and $M_{\tuple{\infty}}$
satisfy conditions (i)-(iv) of Lemma \ref{idealsconstruction}. The
result then follows by applying that lemma.

First, for each $\gamma\in\kappa$, $f\in\quotoel{I}{f_\gamma}$ if
and only if
\[
    \zero(f)\cup Z_\gamma\supset \bigcap_{k=1}^n \left(Z_{\alpha_k}\cup
    Z_{\beta_k}\right),
\]
where, for each $k$, $\alpha_k\not\ll\beta_k$ and
$\beta_k\not\ll\alpha_k$. We see that, for each $k$, one of the
following three cases must happen: (1) $\alpha_k\not\ll\gamma$ and
$\gamma\not\ll\alpha_k$, (2) $\gamma\not\ll\beta_k$ and
$\beta_k\not\ll\gamma$, (3) $\alpha_k\ll\gamma$,
$\beta_k\ll\gamma$, $\alpha_k\not\ll\beta_k$ and
$\beta_k\not\ll\alpha_k$. Thus, we see that
\[
    \zero(f)\cup Z_\gamma\supset \bigcap_{i=1}^r Z_{\varrho_i}\cap
    \bigcap_{j=1}^s \left(Z_{\sigma_j}\cup Z_{\varsigma_j}\right),
\]
where, for each $i$, $\varrho_i\not\ll\gamma$ and
$\gamma\not\ll\varrho_i$, and for each $j$, $\sigma_j\ll\gamma$,
$\varsigma_j\ll\gamma$, $\sigma_j\not\ll\varsigma_j$ and
$\varsigma_j\not\ll\sigma_j$. It follows that Lemma \ref{idealsconstruction}(i) holds:
$f_\gamma\notin I$ and $\quotoel{I}{f_\gamma}\subset
M_{\tuple{\infty}}$ ($\gamma\in\kappa$). It is easily seen that conditions
(ii) and (iii) of Lemma
\ref{idealsconstruction} are satisfied by the
definitions of $I$ and $Z_\alpha$'s.

Now, let $\beta_0,\alpha\in\kappa$ and let $g\in\C(\Xi)$ be such
that $l(\beta_0)=0$, $\beta_0\ll\alpha$, $\beta_0\neq\alpha$ and
$gf_{\beta_0}\in I$ then, from the previous discussion, noting
that $l(\beta_0)=0$, we have
\[
    \zero(g)\cup Z_{\beta_0}\supset \bigcap_{i=1}^r Z_{\varrho_i};
\]
where, for each $i$, $\beta_0\not\ll\varrho_i$ (and so
$\alpha\not\ll\varrho_i$). This implies that
\[
    \zero(g)\supset \bigcap_{i=1}^r Z_{\varrho_i}\cap
    \bigcap_{l(\gamma)=0,\, w(\gamma)\le w(\beta_0)} Z_\gamma.
\]
Without loss of generality, we assume that $\varrho_i\ll\alpha$
($1\le i\le k$) and $\varrho_i\not\ll\alpha$ ($k<i\le r$). Choose
$\beta_1$ to be the smallest element in $\kappa$ with respect to
the properties that $l(\beta_1)=0$, $\beta_1<\alpha$,
$\beta_1>\max\set{\varrho_i: 1\le i\le k}$, and that
\[
    \beta_1>\max\set{\gamma\in\kappa_0:\ \gamma\ll\alpha\ \textrm{and}\ w(\gamma)\le w(\beta_0)}
\]
(it is easy to see that there exists such an element). It follows
that $\beta_1\ll\alpha$. Let $\beta\in\kappa_0$ be such that
$\beta_1\le\beta\le\alpha$. We see that $\beta\ll\alpha$, and so
$\varrho_i\not\ll\beta$ (otherwise, $\varrho_i=\beta\ll\alpha$),
and also $\beta\not\ll\varrho_i$ (otherwise, either
$\varrho_i\ll\alpha$ or $\alpha\ll\varrho_i$) for $1\le i\le r$.
Obviously, $\beta\not\ll\gamma$ and $\gamma\not\ll\beta$ for each
$\gamma\in\kappa_0$ with $w(\gamma)\le w(\beta_0)$. Thus
\[
    \zero(g)\cup Z_\beta\supset \bigcap_{i=1}^r \left(Z_{\varrho_i}\cup Z_\beta\right)
    \cap\bigcap_{l(\gamma)=0,\, w(\gamma)\le w(\beta_0)}  \left(Z_{\gamma}\cup Z_\beta\right)\in \F;
\]
and so $gf_\beta\in I$. Hence, Lemma \ref{idealsconstruction} (iv)
holds. This finishes the proof.
\end{proof}

\begin{remark}
Since the cardinality of $\C(\Xi)$ is $\continuum$, we see that
\[
    \cardinal{\C_0(\Omega)\bigg/\bigcap_{\alpha\in\kappa} P_\alpha}=\continuum.
\]
\end{remark}

\begin{proposition}
    The family of prime ideals $\set{P_\alpha:\ \alpha\in\kappa}$
    constructed in Theorem \ref{examples} is compact. If moreover
    the depth $d$ of $\kappa$ is bigger than $1$, then
    $\bigcap_{\alpha\in\kappa} P_\alpha$ is not the intersection
    of any family that can be decomposed into only finitely many pseudo-finite subfamilies of prime
    ideals.
\end{proposition}
\begin{proof}
Consider the order topology on $\kappa$. Then $\kappa$ is a
compact metrizable space. Each sequence in $\kappa$ has a
convergence subsequence. The following claim will show the
compactness of $\set{P_\alpha:\ \alpha\in\kappa}$.

\emph{Claim: For each sequence $(\alpha_n)$ converging to
$\alpha$, there exists $n_0$ such that $(P_{\alpha_n}:\ n\ge n_0)$
is a pseudo-finite sequence whose union is $P_\alpha$. Proof:}
Without loss of generality, we assume that $\alpha_n\neq \alpha$
($n\in\naturals$). There exists $n_0$ such that
$\alpha_n\ll\alpha$ ($n\ge n_0$). We see that $P_{\alpha_n}\subset
P_\alpha$ ($n\ge n_0$) and
\[
    P_\alpha=\bigcup_{\beta\ll\alpha,\ l(\beta)=0} P_\beta.
\]
So, for each $g\in P_\alpha$, by Theorem \ref{examples}(c), there
exists $\beta_1\ll\alpha$ such that $l(\beta_1)=0$ and that $g\in
P_\beta$ for all $\beta_1\le\beta\le\alpha$ with $l(\beta)=0$.
Choose $n_1$  such that $\alpha_n\ge \beta_1$ ($n\ge n_1$). For
each $n\ge n_1$, pick $\beta'\ll\alpha_n$ and $l(\beta')=0$. If
$\beta'<\beta_1$, then $\beta_1\ll\alpha_n$, and so $g\in
P_{\beta_1}\subset P_{\alpha_n}$. Otherwise, $\beta_1\le\beta'$,
then $g\in P_{\beta'}\subset P_{\alpha_n}$. Thus, $(P_{\alpha_n}:\
n\ge n_0)$ is a pseudo-finite sequence whose union is $P_\alpha$.

Suppose that $I=\bigcap_{\alpha\in\kappa} P_\alpha$ is the
intersection of a family $\setofprimes$ that can be decomposed
into finitely many pseudo-finite subfamilies of prime ideals. By
keeping only the minimal elements of $\setofprimes$, we can assume
that $P\nsubset Q$ for each $P\neq Q\in\setofprimes$. Lemma
\ref{nonredundancy} shows that, for each $P\in\setofprimes$ there
exists $f_P\notin P$ but $f_P\in Q$ for all
$Q\in\setofprimes\setminus\set{P}$. This and Lemma
\ref{equivalent_primes} then implies that $\setofprimes$ is the
set of prime ideals of the form $\quotoel{I}{f}$ for some
$f\in\C_0(\Omega)$. Similarly, $\set{P_\alpha:\ \alpha\in\kappa,\
l(\alpha)=0}$ is the set of prime ideals of the form
$\quotoel{I}{f}$ for some $f\in\C_0(\Omega)$. Thus
\[
    \setofprimes=\set{P_\alpha:\ \alpha\in\kappa,\ l(\alpha)=0},
\]
and obviously this cannot be the union of any finite number of
pseudo-finite families when the depth of $\kappa$ is bigger than
$1$.
\end{proof}


\providecommand{\bysame}{\leavevmode\hbox
to3em{\hrulefill}\thinspace}
\providecommand{\MR}{\relax\ifhmode\unskip\space\fi MR }
\providecommand{\MRhref}[2]{%
  \href{http://www.ams.org/mathscinet-getitem?mr=#1}{#2}
} \providecommand{\href}[2]{#2}

\end{document}